\documentclass{emsprocart}

\allowdisplaybreaks

\usepackage{epic,eepic} % for figures
\usepackage[mathscr]{eucal} %  use \EuScript (\mathcal unchanged)
\numberwithin{equation}{section}
\newtheorem{Theorem}{Theorem}[section]
\newtheorem{Proposition}[Theorem]{Proposition}
\newtheorem{Lemma}[Theorem]{Lemma}

\newtheorem{Conjecture}[Theorem]{Conjecture}
\newtheorem{Problem}[Theorem]{Problem}
\theoremstyle{definition}
\newtheorem{Definition}[Theorem]{Definition}
\newtheorem{Example}[Theorem]{Example}
\newtheorem{Remark}[Theorem]{Remark}

\title[Periodicities in cluster algebras and dilogarithm identities]
{Periodicities in cluster algebras and dilogarithm identities
}

\author{Tomoki Nakanishi}
\contact[nakanisi@math.nagoya-u.ac.jp]{Tomoki Nakanishi,
 Graduate School of Mathematics, Nagoya University,
Nagoya, 464-8604, Japan}

\begin{document}

\begin{abstract}
We consider two kinds of periodicities of mutations in cluster algebras.
For any sequence of mutations under which exchange matrices are periodic,
we define the associated T- and Y-systems.
When the sequence is `regular',
they are particularly
natural generalizations of the known `classic' T- and Y-systems.
Furthermore, for any  sequence of mutations under which 
seeds are periodic,
we formulate the associated dilogarithm identity.
We prove the identities when  exchange matrices are skew symmetric.
\end{abstract}



\maketitle
%\setcounter{tocdepth}{1}
%\tableofcontents[sections]

%\keywords{T-systems; Y-systems;  quantum groups; cluster algebras}

\section{Introduction}

Cluster algebras were introduced by Fomin and Zelevinsky \cite{Fomin02}.
They naturally appear in several different areas of mathematics,
for example, in geometry of surfaces, in coordinate rings of
 algebraic varieties related to Lie groups,
in the representation theory of algebras,
and also in the representation theory of quantum groups, etc.
See \cite{Fock05,Gekhtman05,Geiss07,Caldero06,Buan06,
DiFrancesco09a}, to name a few.

The simplest and the most tractable class of cluster algebras
are the cluster algebras with finitely many  seeds.
They are called the {\em cluster algebras of finite type\/}
and play a fundamental role in many applications.
Fomin and Zelevinsky classified the cluster algebras
of finite type by the Dynkin diagrams
in their pioneering works \cite{Fomin03a,Fomin03b}.
They also clarified  the intimate relation between
 cluster algebras of finite type
 and root systems of finite type.
In particular, a remarkable periodicity property of mutations of seeds
was discovered and proved; it is related to the Coxeter elements
of the Weyl groups
\cite{Fomin03a,Fomin03b,Fomin07}.

In this paper we focus on two kinds of periodicities of
mutations in general cluster algebras.
The first one is the periodicity of {\em exchange matrices (or quivers)\/}
under a sequence of mutations.
In other words, the
 {\em exchange relations\/} of clusters and coefficient tuples are
periodic
under such a sequence of mutations.
The second one is the periodicity of {\em seeds\/}
under a sequence of mutations.
The latter periodicity implies the former one,
but the converse is not true.

Let us briefly explain the background for this study.
Many examples of such periodicities 
appeared in connection with
systems of algebraic relations called {\em T-systems and Y-systems\/}
\cite{Zamolodchikov91,Klumper92,Kuniba92,Ravanini93,Kuniba94a,
Hernandez07a}
and the 
 {\em dilogarithm identities}
\cite{Kirillov86,Kirillov89,Bazhanov90,Kirillov90,Kuniba93a,Gliozzi95}
which originated in the study of integrable models in two dimensions.
In retrospect, they are part of the cluster algebraic
structure which appeared prior to the 
 notion of cluster algebra itself.
Naturally and inevitably,
their cluster algebraic nature has been gradually revealed recently
\cite{Fomin03b,Chapoton05,Fomin07, Keller08, DiFrancesco09a,Hernandez09, Kuniba09,
Inoue10c, Nakajima09,Keller10,Nakanishi09,Inoue10a,
Inoue10b,Nakanishi10a,Nakanishi10b},
and it turned out that
the cluster algebraic formulation and machinery
are very powerful and essential to understanding their properties.

Since examples of periodicities of exchange relations
and seeds are accumulating,
it may be a good time to reverse the viewpoint,
namely, to formulate T-systems,
Y-systems, and dilogarithm identities
in a more general and unified setting,
starting from general cluster algebras with such periodicity properties.
This is the subject of the paper.

Let us summarize the main result of the paper.
For any  sequence of mutations
under which exchange matrices are periodic,
we define the associated T- and Y-systems.
When the sequence is `regular',
they are particularly natural generalizations of the known `classic' T-
and Y-systems.
The definition of the term `regular' is found in Definition
\ref{def:regular}.
Furthermore, for any  
sequence of mutations
under which seeds are periodic,
we formulate the associated dilogarithm identity.
We prove the identities when  exchange matrices are skew symmetric.
We expect that there will be several applications of
the result in various areas related to cluster algebras.

We mention that many examples of
periodicities of exchange matrices
were constructed by Fordy and Marsh \cite{Fordy09}
and the associated T-systems were also introduced.

The relation between cluster algebras and the dilogarithm
was also studied earlier by Fock and Goncharov \cite{Fock09,Fock07}.
In fact, Proposition \ref{prop:local} is motivated by
their formulas.
However, there is a subtle but important difference;
that is, we use {\em $F$-polynomials} of \cite{Fomin07}
in \eqref{eq:W}.
%Therefore, in spite of overlap with the results with 
%\cite{Fock09,Fock07},
%we believe that the presentation in Section \ref{sec:dilog} of this
%paper still contains another useful information to clarify the relation
%between cluster algebras and the dilogarithm function.
See Section \ref{subsec:local} for more details.

The organization of the paper is as follows.
In Section 2 we introduce the basic notions for
periodicities of exchange matrices and seeds.
In Section 3 we present some of known examples of periodicities
of exchange matrices (or quivers) and seeds,
 most of which are connected to `classic' T- and Y-systems.
In Section 4 we give Restriction/Extension Theorem of
periodicities of seeds.
After these preparations, 
in Section 5,
for any sequence of mutations
under which exchange matrices are periodic,
we define the associated T- and Y-systems.
Here we do not require the periodicity of seeds.
Special attention is paid to the case when the sequence is regular.
In Section 6,
for any  sequence of mutations
under which seeds are periodic,
we formulate the associated dilogarithm identity.
Then, we prove the identities when  exchange matrices are skew symmetric 
in Theorem \ref{thm:DI},
which is the main theorem of the paper.

\section*{Acknowledgments}
I thank Rei Inoue, Osamu Iyama, Bernhard Keller,
Atsuo Kuniba, Roberto Tateo, and Junji Suzuki
for sharing their insights in the preceding joint works.
I am grateful to Fr\'ed\'eric Chapoton, Vladimir Fock,
Bernhard Keller, Robert Marsh, Pierre-Guy
Plamondon, and Andrei Zelevinsky
 for kindly communicating their works
and also for useful comments.
Finally, I thank Andrzej Skowronski for his kind
invitation to this proceedings volume.

\section{Periodicities of exchange matrices and seeds}

\subsection{Cluster algebras with coefficients}
\label{subsec:cluster}

In this subsection
we recall the definition
of the cluster algebras with
coefficients
and some of their basic properties,
following the convention in \cite{Fomin07}
with slight change of notations and terminology.
See \cite{Fomin07} for more details and information.

Fix an arbitrary  semifield $\mathbb{P}$,
i.e., an abelian multiplicative group
endowed with a binary
operation of addition $\oplus$ which is commutative,
associative, and distributive with respect to the
multiplication \cite{Hutchins90}.
Let $\mathbb{Q}\mathbb{P}$
denote the quotient field of the group ring $\mathbb{Z}\mathbb{P}$
of $\mathbb{P}$.
Let $I$ be a finite set, and let $B=(b_{ij})_{i,j\in I}$ be a
 skew
symmetrizable (integer) matrix; namely, there is a diagonal positive integer
matrix $D$ such that
${}^{t}(DB)=-DB$.
Let $x=(x_i)_{i\in I}$
be an $I$-tuple of formal variables,
and let $y=(y_i)_{i\in I}$ be an $I$-tuple of elements in $\mathbb{P}$.
For the triplet $(B,x,y)$, called the {\em initial seed},
the  {\em cluster algebra $\mathcal{A}(B,x,y)$ with
coefficients in $\mathbb{P}$} is defined as follows.

Let $(B',x',y')$ be a triplet consisting of
skew symmetrizable matrix $B'=(b'_{ij})_{i,j\in I}$,
an $I$-tuple $x'=(x'_i)_{i\in I}$ with
 $x'_i\in \mathbb{Q}\mathbb{P}(x)$,
and 
an $I$-tuple $y'=(y'_i)_{i\in I}$ with $y'_i\in \mathbb{P}$.
For each $k\in I$, we define another triplet
$(B'',x'',y'')=\mu_k(B',x',y')$, called the {\em mutation
of $(B',x',y')$ at $k$}, as follows.

{\it  (i) Mutation of matrix.}
\begin{align}
\label{eq:Bmut}
b''_{ij}=
\begin{cases}
-b'_{ij}& \mbox{$i=k$ or $j=k$},\\
b'_{ij}+\frac{1}{2}
(|b'_{ik}|b'_{kj} + b'_{ik}|b'_{kj}|)
&\mbox{otherwise}.
\end{cases}
\end{align}

{\it (ii)  Exchange relation of coefficient tuple.}
\begin{align}
\label{eq:coef}
y''_i =
\begin{cases}
\displaystyle
{y'_k}{}^{-1}&i=k,\\
\displaystyle
y'_i \frac{1}{(1\oplus {y'_k}^{-1})^{b'_{ki}}}&
i\neq k,\ b'_{ki}\geq 0,\\
y'_i (1\oplus y'_k)^{-b'_{ki}}&
i\neq k,\ b'_{ki}\leq 0.\\
\end{cases}
\end{align}

{\it (iii)   Exchange relation of cluster.}
\begin{align}
\label{eq:clust}
x''_i =
\begin{cases}
\displaystyle
\frac{y'_k
\prod_{j: b'_{jk}>0} {x'_j}^{b'_{jk}}
+
\prod_{j: b'_{jk}<0} {x'_j}^{-b'_{jk}}
}{(1\oplus y'_k)x'_k}
&
i= k,\\
{x'_i}&i\neq k.\\
\end{cases}
\end{align}
It is easy to see that  $\mu_k$ is an involution,
namely, $\mu_k(B'',x'',y'')=(B',x',y')$.
Now, starting from the initial seed
$(B,x,y)$, iterate mutations and collect all the
resulting triplets $(B',x',y')$.
We call $(B',x',y')$ a {\em seed\/},
$y'$ and $y'_i$ a {\em coefficient tuple} and
a {\em coefficient\/},
$x'$  and $x'_i$, a {\em cluster\/} and
a {\em cluster variable}, respectively.
The {\em cluster algebra $\mathcal{A}(B,x,y)$ with
coefficients in $\mathbb{P}$} is the
$\mathbb{Z}\mathbb{P}$-subalgebra of the
rational function field $\mathbb{Q}\mathbb{P}(x)$
generated by all the cluster variables.
Similarly, the {\em coefficient group $\mathcal{G}(B,y)$ with
coefficients in $\mathbb{P}$} is the
multiplicative subgroup of the semifield $\mathbb{P}$
generated by all the coefficients $y'_i$ together with $1\oplus y'_i$.

It is standard to identify
a {\em skew symmetric} (integer) matrix $B=(b_{ij})_{i,j\in I}$
with a {\em quiver $Q$
without loops or 2-cycles}.
The set of the vertices of $Q$ is given by $I$,
and we put $b_{ij}$ arrows from $i$ to $j$ 
if $b_{ij}>0$.
The mutation $Q''=\mu_k(Q')$ of a quiver $Q'$ is given by the following
rule:
For each pair of an incoming arrow $i\rightarrow k$
and an outgoing arrow $k\rightarrow j$ in $Q'$,
add a new arrow $i\rightarrow j$.
Then, remove a maximal set of pairwise disjoint 2-cycles.
Finally, reverse all arrows incident with $k$.

Let  $\mathbb{P}_{\mathrm{univ}}(y)$ 
be the {\em universal semifield\/} of
the $I$-tuple of generators $y=(y_i)_{i\in I}$, namely,
the semifield consisting of 
the {\em subtraction-free\/} rational functions of formal
variables $y$ with
usual multiplication and addition in the rational function
 field $\mathbb{Q}(y)$.
We write $\oplus$ in $\mathbb{P}_{\mathrm{univ}}(y)$ as $+$
for simplicity when it is not confusing.

{}{\em From now on, unless otherwise mentioned,
we set the semifield $\mathbb{P}$ for $\mathcal{A}(B,x,y)$
to be $\mathbb{P}_{\mathrm{univ}}(y)$,
where $y$ is the coefficient tuple in the initial seed $(B,x,y)$.}

Let $\mathbb{P}_{\mathrm{trop}}(y)$ 
 be the {\em tropical semifield\/} 
of $y=(y_i)_{i\in I}$, which
is the abelian multiplicative group freely generated by
$y$ endowed with the addition $\oplus$
\begin{align}
\label{eq:trop}
\prod_i y_i^{a_i}\oplus
\prod_i y_i^{b_i}
=
\prod_i y_i^{\min(a_i,b_i)}.
\end{align}
There is a canonical surjective semifield homomorphism
$\pi_{\mathbf{T}}$ (the {\em tropical evaluation})
from $\mathbb{P}_{\mathrm{univ}}(y)$
to $\mathbb{P}_{\mathrm{trop}}(y)$ defined by $\pi_{\mathbf{T}}(y_i)= y_i$
and $\pi_{\mathbf{T}}(\alpha)=1$ ($\alpha \in \mathbb{Q}_+$).
For any coefficient $y'_i$ of $\mathcal{A}(B,x,y)$,
let us write $[y'_i]_{\mathbf{T}}:= \pi_{\mathbf{T}}(y'_i)$ for simplicity.
We call $[y'_i]_{\mathbf{T}}$'s the {\em tropical coefficients\/}
(the {\em principal coefficients\/} in \cite{Fomin07}).
They satisfy the exchange relation \eqref{eq:coef}
by replacing $y'_i$ with $[y'_i]_{\mathbf{T}}$ 
with $\oplus$ being the addition in \eqref{eq:trop}.
We also extend this homomorphism to
the homomorphism of fields
$\pi_{\mathbf{T}}:(\mathbb{Q}\mathbb{P}_{\mathrm{univ}}(y))(x)
\rightarrow 
(\mathbb{Q}\mathbb{P}_{\mathrm{trop}}(y))(x)$.

To each seed $(B',x',y')$ of $\mathcal{A}(B,x,y)$
we attach the {\em $F$-polynomials\/} $F'_i(y)\in
 \mathbb{Q}(y)$ ($i\in I$)
by the specialization of $[x'_i]_{\mathbf{T}}$
at $x_j=1$ ($j\in I$).
It is, in fact, a polynomial in $y$ with integer coefficients
due to the Laurent phenomenon \cite[Proposition 3.6]{Fomin07}.
For definiteness, let us take  $I=\{1,\dots,n\}$.
Then,
$x'$ and $y'$ have the following factorized expressions
\cite[Proposition 3.13, Corollary 6.3]{Fomin07}
by the $F$-polynomials.
\begin{align}
\label{eq:gF}
x'_i &=
\left(
\prod_{j=1}^n x_j^{g'_{ji}}
\right)
\frac{
F'_i(\hat{y}_1, \dots,\hat{y}_n)
}
{
F'_i(y_1, \dots,y_n)
},
\quad
\hat{y}_i=y_i\prod_{j=1}^n x_j^{b_{ji}},
\\
\label{eq:Yfact}
y'_i&=
[y'_i]_{\mathbf{T}}
\prod_{j=1}^n F'_j(y_1,\dots,y_n)^{b'_{ji}}.
\end{align}
The integer vector $\mathbf{g}'_i=(g'_{1i},\dots,g'_{ni})$ ($i=1,\dots,n$)
uniquely determined by \eqref{eq:gF}
for each $x'_i$ is called the {\em $g$-vector}
for $x'_i$.

\begin{Conjecture}[\cite{Fomin07}]
\label{conj:pos}
For any cluster algebra $\mathcal{A}(B,x,y)$
with  skew symmetrizable matrix $B$,
the following properties hold.
\par 
(a) Each  tropical coefficient $[y'_i]_{\mathbf{T}}$ 
is not $1$, and, either positive or negative Laurent
monomial in $y$.
\par
(b) (the `sign coherence')
For each $i$, the $i$th components
of $g$-vectors, $g'_{ij}$ ($j=1,\dots,n+1$) in \eqref{eq:gF},
are simultaneously nonpositive or nonnegative.
\par
(c) Each $F$-polynomial $F'_i(y)$ has a constant term 1.
\end{Conjecture}

This conjecture was  proved
in the skew symmetric case
 by \cite{Derksen10, Plamondon10b, Nagao10}
 with the result of \cite[Proposition 5.6]{Fomin07}.
\begin{Theorem}
\label{thm:pos}
Conjecture \ref{conj:pos} is true
for any skew symmetric matrix $B$.
\end{Theorem}

Let $\mathbf{i}=(i_1,\dots,i_r)$ be an $I$-sequence,
namely, $i_1,\dots,i_r\in I$.
We define the {\em composite mutation\/} $\mu_{\mathbf{i}}$
by $\mu_{\mathbf{i}}=\mu_{i_r}
\cdots \mu_{i_2} \mu_{i_1}$, where the product means
the composition.
For $I$-sequences $\mathbf{i}$ and $\mathbf{i}'$,
we write $\mathbf{i}{\sim}_B
\mathbf{i}'$ if $\mu_{\mathbf{i}}(B,x,y)=
\mu_{\mathbf{i}'}(B,x,y)$.

The following fact will be used implicitly and frequently.
\begin{Lemma}
\label{lem:order}
Let $B=(b_{ij})_{i,j\in I}$ be a skew symmetrizable matrix
and let $\mathbf{i}=(i_1,\dots,i_r)$ be an $I$-sequence.
Suppose that $b_{{i_a}{i_b}}=0$ for any $1\leq a,b \leq r$.
Then, the following facts hold.
\par
(a) For any permutation $\sigma$ of
$\{1,\dots, r\}$, we have
\begin{align}
{\mathbf{i}}\sim_B
{(i_{\sigma(1)},\dots,i_{\sigma(r)})}.
\end{align}
\par
(b)
Let $B'=\mu_{\mathbf{i}}(B)$.
Then, $b'_{{i_a}{i_b}}=0$ holds for any $1\leq a,b \leq r$.
\par
(c)
Let  $(B',x',y')=\mu_{\mathbf{i}}(B,x,y)$.
Then, $(B,x,y)=\mu_{\mathbf{i}}(B',x',y')$.
\end{Lemma}
\begin{proof}
The facts (a) and (b) are easily verified from
\eqref{eq:Bmut}--\eqref{eq:clust}.
The fact (c) follows from (a), (b), and the involution property of
each mutation $\mu_i$.
\end{proof}

\subsection{Periodicities of exchange matrices and seeds}

\begin{Definition}
\label{def:period}
Let $\mathcal{A}(B,x,y)$ be a cluster algebra,
$(B',x',y')$ be a seed of $\mathcal{A}(B,x,y)$,
$\mathbf{i}=(i_1,\dots,i_r)$ be
an $I$-sequence,
$(B'',x'',y'')=\mu_{\mathbf{i}}(B',x',y')$,
and
$\nu:I \rightarrow I$ be a bijection.
\par
(a) We call the sequence  $\mathbf{i}$
 a {\em $\nu$-period\/} of $B'$ if
$b''_{\nu(i)\nu(j)}=b'_{ij}$ ($i,j\in I$) holds;
furthermore,
 if $\nu=\mathrm{id}$, we simply call it a {\em period} of $B'$.
\par
(b)
We call the sequence  $\mathbf{i}$ a {\em $\nu$-period\/} of $(B',x',y')$
if
\begin{align}
\label{eq:periodBxy}
b''_{\nu(i)\nu(j)}=b'_{ij},
\quad
x''_{\nu(i)}=x'_i,
\quad
y''_{\nu(i)}=y'_i
\quad (i,j\in I)
\end{align}
 holds;
furthermore,
 if $\nu=\mathrm{id}$, we simply call it a {\em period} of $(B',x',y')$.
\end{Definition}

For any seed $(B',x',y')$ of  $\mathcal{A}(B,x,y)$,
$\mathcal{A}(B',x',y')$ is isomorphic to
$\mathcal{A}(B,x,y)$ as a cluster algebra.
Therefore, by resetting the initial seed if necessary,
we may concentrate on the situation
where $(B',x',y')=(B,x,y)$ in the above
without losing generality.

If $\mathbf{i}$ is a $\nu$-period of $(B,x,y)$,
then, of course, it is a $\nu$-period of $B$.
However, the converse does not hold, in general.

If $\mathbf{i}$ is a $\nu$-period of $B$ and
there is a nontrivial
automorphism $\omega:I\rightarrow I$ of $B$,
i.e., $b_{\omega(i)\omega(j)}=b_{ij}$,
then $\mathbf{i}$ is also an $\nu\omega$-period.
On the other hand, there is no such ambiguity for
a $\nu$-period of $(B,x,y)$, since
cluster variables $x_i$ $(i\in I)$ are algebraically independent.

If   $\mathbf{i}$ and  $\mathbf{i}'$
are a $\nu$-period and a $\nu'$-period of $B$, or $(B,x,y)$,
respectively,
then the concatenation of sequences
\begin{align}
 \mathbf{i}\,|\, \nu(\mathbf{i}')
:=(i_1,\dots,i_r, \nu(i'_1),\dots,\nu(i'_{r'}))
\end{align}
is a $\nu\nu'$-period of $B$, or $(B,x,y)$.
In particular,  $\mathbf{i}\,|\, \nu(\mathbf{i})\,|\, \cdots
\,|\, \nu^{p-1}(\mathbf{i})$ is a $\nu^p$-period of $B$,
or $(B,x,y)$, for any positive integer $p$.
Since $\nu$ acts on a finite set $I$,
it has a finite order, say, $g$.
Define
\begin{align}
\mathbf{j}(\mathbf{i},\nu):=
\mathbf{i}\,|\, \nu(\mathbf{i})\,|\, \cdots
\,|\, \nu^{g-1}(\mathbf{i}).
\end{align}
Then, $\mathbf{j}(\mathbf{i},\nu)$ is a period of $B$, or $(B,x,y)$.

Examples of periodicities of exchange matrices (or quivers)
and seeds
will be given in Section \ref{sec:examples}.

\begin{Proposition}
\label{prop:opposite}
 If
 $\mathbf{i}$ is a period
of $(B,x,y)$ in $\mathcal{A}(B,x,y)$,
then 
 $\mathbf{i}$ is also a period
of $(-B,x,y)$ in $\mathcal{A}(-B,x,y)$.
\end{Proposition}
\begin{proof}
This is due to the  duality 
between the exchange $B\leftrightarrow -B$,
$y\leftrightarrow y^{-1}$.
Namely, the correspondence of seeds $(B',x',y')$ in $\mathcal{A}(B,x,y)$
and $(-B',x',y'^{-1})$ in $\mathcal{A}(-B,x,y^{-1})$
commutes with mutations and yields
the isomorphism of cluster algebras.
\end{proof}

When $B$ is skew symmetric,
the transformation $B\leftrightarrow -B$ corresponds
to the transformation $Q\leftrightarrow Q^{\mathrm{op}}$,
where $Q^{\mathrm{op}}$ is the opposite quiver of $Q$.

\subsection{Criterion of periodicity of seeds for skew symmetric case}

In general, checking
the condition \eqref{eq:periodBxy} directly is a {\em very\/} difficult task.
However, at least when $B$ is {\em skew symmetric},
one can reduce the condition drastically,
thanks to the existence of the categorification with
2-Calabi-Yau property by Plamondon \cite{Plamondon10a,Plamondon10b}.

\begin{Theorem}[{\cite{Plamondon10a,Plamondon10b},
\cite[Theorem 5.1]{Inoue10a}}]
\label{thm:tropperiod}
Assume that the matrix $B$ in Definition \ref{def:period}
is skew symmetric.
Then, the condition \eqref{eq:periodBxy} holds if and only if
the following condition holds:
\begin{align}
\label{eq:periody}
[y''_{\nu(i)}]_{\mathbf{T}}=[y'_i]_{\mathbf{T}}
\quad (i\in I).
\end{align}
\end{Theorem}

We expect that
Theorem \ref{thm:tropperiod}
holds for any skew symmetrizable matrix $B$.

Let us also mention that
the application software by
Bernhard Keller
\cite{Keller08c} is a practical and versatile
tool to check and explore periodicities of
quivers and seeds.

\section{Examples}
% of weakly periodic and periodic cluster algebras}
\label{sec:examples}

We present some of known examples of periodicities
of exchange matrices (or quivers) and seeds.

\subsection{Examples of periodicities of exchange matrices}

There are plenty of examples of periodicities of exchange matrices.

We identify a skew symmetric matrix $B$ and the corresponding
quiver $Q$ as in Section \ref{subsec:cluster}.

\begin{Example}
\label{ex:preperiod}
{\em Cluster algebras for bipartite matrices with alternating property
\cite{Fomin07}.}

 Assume that a skew symmetrizable matrix $B$ is {\em bipartite\/};
namely,  the index set $I$ of $B$
 admits the decomposition
$I=I_+\sqcup I_-$ such that, for any pair $(i,j)$ with $b_{ij}\neq 0$,
either $i\in I_+$, $j\in I_-$ or $i\in I_-$, $j\in I_+$ holds.
Assume further that $B$ has the following `alternating' property:
$b_{ij}>0$ only if $i\in I_+$, $j\in I_-$.
(In the quiver picture, $i\in I_+$ is a source
and $i\in I_-$ is a sink so that the quiver is alternating.)
Let $\mathbf{i}_+$ and $\mathbf{i}_-$ be the sequences of
all the distinct elements of $I_+$ and $I_-$,
respectively,
where the order of the sequence is chosen arbitrarily
thanks to Lemma \ref{lem:order}.
Then, $\mu_{\mathbf{i}_+}(B) =-B$
and $\mu_{\mathbf{i}_-}(-B) = B$.
Thus, $\mathbf{i}=\mathbf{i}_+ \, | \,
\mathbf{i}_-$ is a period of $B$.
The associated Y-system and `T-system' (in our terminology)
were studied in detail in \cite{Fomin07}.
(A matrix $B$ here is called a `bipartite matrix' in \cite{Fomin07}.)
\end{Example}

%%%%%%%%%%%%%%%%%%%%%%%%%%%%%
% G_2
\begin{figure}
\begin{center}
\setlength{\unitlength}{0.71pt}
\begin{picture}(300,288)(0,-15)
%
% first diagram
\put(0,0)
{
\put(0,60){\circle{5}}
\put(0,135){\circle{5}}
\put(0,210){\circle{5}}
\put(30,0){\circle*{5}}
\put(30,15){\circle*{5}}
\put(30,30){\circle*{5}}
\put(30,45){\circle*{5}}
\put(30,60){\circle*{5}}
\put(30,75){\circle*{5}}
\put(30,90){\circle*{5}}
\put(30,105){\circle*{5}}
\put(30,120){\circle*{5}}
\put(30,135){\circle*{5}}
\put(30,150){\circle*{5}}
\put(30,165){\circle*{5}}
\put(30,180){\circle*{5}}
\put(30,195){\circle*{5}}
\put(30,210){\circle*{5}}
\put(30,225){\circle*{5}}
\put(30,240){\circle*{5}}
\put(30,255){\circle*{5}}
\put(30,270){\circle*{5}}
%
% vertical arrows
\put(30,3){\vector(0,1){9}}
\put(30,27){\vector(0,-1){9}}
\put(30,33){\vector(0,1){9}}
\put(30,57){\vector(0,-1){9}}
\put(30,63){\vector(0,1){9}}
\put(30,87){\vector(0,-1){9}}
\put(30,93){\vector(0,1){9}}
\put(30,117){\vector(0,-1){9}}
\put(30,123){\vector(0,1){9}}
\put(30,147){\vector(0,-1){9}}
\put(30,153){\vector(0,1){9}}
\put(30,177){\vector(0,-1){9}}
\put(30,183){\vector(0,1){9}}
\put(30,207){\vector(0,-1){9}}
\put(30,213){\vector(0,1){9}}
\put(30,237){\vector(0,-1){9}}
\put(30,243){\vector(0,1){9}}
\put(30,267){\vector(0,-1){9}}
\put(0,132){\vector(0,-1){69}}
\put(0,138){\vector(0,1){69}}
% diagonal arrows
\put(3,51){\vector(1,-2){24}}
\put(27,19){\vector(-2,3){23}}
\put(3,58){\vector(1,-1){24}}
\put(27,47){\vector(-2,1){23}}
\put(3,60){\vector(1,0){24}}
\put(27,73){\vector(-2,-1){23}}
\put(3,62){\vector(1,1){24}}
\put(27,101){\vector(-2,-3){23}}
\put(3,69){\vector(1,2){24}}
\put(27,135){\vector(-1,0){24}}
\put(3,201){\vector(1,-2){24}}
\put(27,169){\vector(-2,3){23}}
\put(3,208){\vector(1,-1){24}}
\put(27,197){\vector(-2,1){23}}
\put(3,210){\vector(1,0){24}}
\put(27,223){\vector(-2,-1){23}}
\put(3,212){\vector(1,1){24}}
\put(27,251){\vector(-2,-3){23}}
\put(3,219){\vector(1,2){24}}
\put(3,-2)
{
\put(-17,60){\small $-$}
\put(-17,135){\small $+$}
\put(-17,210){\small $-$}
}
\put(4,-2)
{
\put(30,0){\small $+$}
\put(30,15){\small $-$}
\put(30,30){\small $+$}
\put(30,45){\small $-$}
\put(30,60){\small $+$}
\put(30,75){\small $-$}
\put(30,90){\small $+$}
\put(30,105){\small $-$}
\put(30,120){\small $+$}
\put(30,135){\small $-$}
\put(30,150){\small $+$}
\put(30,165){\small $-$}
\put(30,180){\small $+$}
\put(30,195){\small $-$}
\put(30,210){\small $+$}
\put(30,225){\small $-$}
\put(30,240){\small $+$}
\put(30,255){\small $-$}
\put(30,270){\small $+$}
}
}
% second diagram
\put(70,0)
{
\put(0,60){\circle{5}}
\put(0,135){\circle{5}}
\put(0,210){\circle{5}}
\put(30,0){\circle*{5}}
\put(30,15){\circle*{5}}
\put(30,30){\circle*{5}}
\put(30,45){\circle*{5}}
\put(30,60){\circle*{5}}
\put(30,75){\circle*{5}}
\put(30,90){\circle*{5}}
\put(30,105){\circle*{5}}
\put(30,120){\circle*{5}}
\put(30,135){\circle*{5}}
\put(30,150){\circle*{5}}
\put(30,165){\circle*{5}}
\put(30,180){\circle*{5}}
\put(30,195){\circle*{5}}
\put(30,210){\circle*{5}}
\put(30,225){\circle*{5}}
\put(30,240){\circle*{5}}
\put(30,255){\circle*{5}}
\put(30,270){\circle*{5}}
%
% vertical arrows
\put(30,3){\vector(0,1){9}}
\put(30,27){\vector(0,-1){9}}
\put(30,33){\vector(0,1){9}}
\put(30,57){\vector(0,-1){9}}
\put(30,63){\vector(0,1){9}}
\put(30,87){\vector(0,-1){9}}
\put(30,93){\vector(0,1){9}}
\put(30,117){\vector(0,-1){9}}
\put(30,123){\vector(0,1){9}}
\put(30,147){\vector(0,-1){9}}
\put(30,153){\vector(0,1){9}}
\put(30,177){\vector(0,-1){9}}
\put(30,183){\vector(0,1){9}}
\put(30,207){\vector(0,-1){9}}
\put(30,213){\vector(0,1){9}}
\put(30,237){\vector(0,-1){9}}
\put(30,243){\vector(0,1){9}}
\put(30,267){\vector(0,-1){9}}
\put(0,63){\vector(0,1){69}}
\put(0,207){\vector(0,-1){69}}
% diagonal arrows
\put(27,19){\vector(-2,3){23}}
\put(3,58){\vector(1,-1){24}}
\put(27,47){\vector(-2,1){23}}
\put(3,60){\vector(1,0){24}}
\put(27,73){\vector(-2,-1){23}}
\put(3,62){\vector(1,1){24}}
\put(27,101){\vector(-2,-3){23}}
\put(3,133){\vector(2,-1){24}}
\put(27,135){\vector(-1,0){23}}
\put(3,137){\vector(2,1){24}}
\put(27,169){\vector(-2,3){23}}
\put(3,208){\vector(1,-1){24}}
\put(27,197){\vector(-2,1){23}}
\put(3,210){\vector(1,0){24}}
\put(27,223){\vector(-2,-1){23}}
\put(3,212){\vector(1,1){24}}
\put(27,251){\vector(-2,-3){23}}
\put(3,-2)
{
\put(-17,60){\small $+$}
\put(-17,135){\small $-$}
\put(-17,210){\small $+$}
}
\put(4,-2)
{
\put(30,0){\small $+$}
\put(30,15){\small $-$}
\put(30,30){\small $+$}
\put(30,45){\small $-$}
\put(30,60){\small $+$}
\put(30,75){\small $-$}
\put(30,90){\small $+$}
\put(30,105){\small $-$}
\put(30,120){\small $+$}
\put(30,135){\small $-$}
\put(30,150){\small $+$}
\put(30,165){\small $-$}
\put(30,180){\small $+$}
\put(30,195){\small $-$}
\put(30,210){\small $+$}
\put(30,225){\small $-$}
\put(30,240){\small $+$}
\put(30,255){\small $-$}
\put(30,270){\small $+$}
}
}
% third diagram
\put(140,0)
{
\put(0,60){\circle{5}}
\put(0,135){\circle{5}}
\put(0,210){\circle{5}}
\put(30,0){\circle*{5}}
\put(30,15){\circle*{5}}
\put(30,30){\circle*{5}}
\put(30,45){\circle*{5}}
\put(30,60){\circle*{5}}
\put(30,75){\circle*{5}}
\put(30,90){\circle*{5}}
\put(30,105){\circle*{5}}
\put(30,120){\circle*{5}}
\put(30,135){\circle*{5}}
\put(30,150){\circle*{5}}
\put(30,165){\circle*{5}}
\put(30,180){\circle*{5}}
\put(30,195){\circle*{5}}
\put(30,210){\circle*{5}}
\put(30,225){\circle*{5}}
\put(30,240){\circle*{5}}
\put(30,255){\circle*{5}}
\put(30,270){\circle*{5}}
%
% vertical arrows
\put(30,3){\vector(0,1){9}}
\put(30,27){\vector(0,-1){9}}
\put(30,33){\vector(0,1){9}}
\put(30,57){\vector(0,-1){9}}
\put(30,63){\vector(0,1){9}}
\put(30,87){\vector(0,-1){9}}
\put(30,93){\vector(0,1){9}}
\put(30,117){\vector(0,-1){9}}
\put(30,123){\vector(0,1){9}}
\put(30,147){\vector(0,-1){9}}
\put(30,153){\vector(0,1){9}}
\put(30,177){\vector(0,-1){9}}
\put(30,183){\vector(0,1){9}}
\put(30,207){\vector(0,-1){9}}
\put(30,213){\vector(0,1){9}}
\put(30,237){\vector(0,-1){9}}
\put(30,243){\vector(0,1){9}}
\put(30,267){\vector(0,-1){9}}
\put(0,132){\vector(0,-1){69}}
\put(0,138){\vector(0,1){69}}
% diagonal arrows
\put(3,58){\vector(1,-1){24}}
\put(27,47){\vector(-2,1){23}}
\put(3,60){\vector(1,0){24}}
\put(27,73){\vector(-2,-1){23}}
\put(3,62){\vector(1,1){24}}
\put(27,107){\vector(-1,1){23}}
\put(3,133){\vector(2,-1){24}}
\put(27,135){\vector(-1,0){23}}
\put(3,137){\vector(2,1){24}}
\put(27,163){\vector(-1,-1){23}}
\put(3,208){\vector(1,-1){24}}
\put(27,197){\vector(-2,1){23}}
\put(3,210){\vector(1,0){24}}
\put(27,223){\vector(-2,-1){23}}
\put(3,212){\vector(1,1){24}}
\put(3,-2)
{
\put(-17,60){\small $-$}
\put(-17,135){\small $+$}
\put(-17,210){\small $-$}
}
\put(4,-2)
{
\put(30,0){\small $+$}
\put(30,15){\small $-$}
\put(30,30){\small $+$}
\put(30,45){\small $-$}
\put(30,60){\small $+$}
\put(30,75){\small $-$}
\put(30,90){\small $+$}
\put(30,105){\small $-$}
\put(30,120){\small $+$}
\put(30,135){\small $-$}
\put(30,150){\small $+$}
\put(30,165){\small $-$}
\put(30,180){\small $+$}
\put(30,195){\small $-$}
\put(30,210){\small $+$}
\put(30,225){\small $-$}
\put(30,240){\small $+$}
\put(30,255){\small $-$}
\put(30,270){\small $+$}
}
}
%
% fourth diagram
\put(210,0)
{
\put(0,60){\circle{5}}
\put(0,135){\circle{5}}
\put(0,210){\circle{5}}
\put(30,0){\circle*{5}}
\put(30,15){\circle*{5}}
\put(30,30){\circle*{5}}
\put(30,45){\circle*{5}}
\put(30,60){\circle*{5}}
\put(30,75){\circle*{5}}
\put(30,90){\circle*{5}}
\put(30,105){\circle*{5}}
\put(30,120){\circle*{5}}
\put(30,135){\circle*{5}}
\put(30,150){\circle*{5}}
\put(30,165){\circle*{5}}
\put(30,180){\circle*{5}}
\put(30,195){\circle*{5}}
\put(30,210){\circle*{5}}
\put(30,225){\circle*{5}}
\put(30,240){\circle*{5}}
\put(30,255){\circle*{5}}
\put(30,270){\circle*{5}}
%
% vertical arrows
\put(30,3){\vector(0,1){9}}
\put(30,27){\vector(0,-1){9}}
\put(30,33){\vector(0,1){9}}
\put(30,57){\vector(0,-1){9}}
\put(30,63){\vector(0,1){9}}
\put(30,87){\vector(0,-1){9}}
\put(30,93){\vector(0,1){9}}
\put(30,117){\vector(0,-1){9}}
\put(30,123){\vector(0,1){9}}
\put(30,147){\vector(0,-1){9}}
\put(30,153){\vector(0,1){9}}
\put(30,177){\vector(0,-1){9}}
\put(30,183){\vector(0,1){9}}
\put(30,207){\vector(0,-1){9}}
\put(30,213){\vector(0,1){9}}
\put(30,237){\vector(0,-1){9}}
\put(30,243){\vector(0,1){9}}
\put(30,267){\vector(0,-1){9}}
\put(0,63){\vector(0,1){69}}
\put(0,207){\vector(0,-1){69}}
%\put(0,132){\vector(0,-1){69}}
%\put(0,138){\vector(0,1){69}}
% diagonal arrows
\put(27,47){\vector(-2,1){23}}
\put(3,60){\vector(1,0){24}}
\put(27,73){\vector(-2,-1){23}}
\put(3,128){\vector(2,-3){24}}
\put(27,107){\vector(-1,1){23}}
\put(3,133){\vector(2,-1){24}}
\put(27,135){\vector(-1,0){23}}
\put(3,137){\vector(2,1){24}}
\put(3,142){\vector(2,3){23}}
\put(27,163){\vector(-1,-1){24}}
\put(27,197){\vector(-2,1){23}}
\put(3,210){\vector(1,0){24}}
\put(27,223){\vector(-2,-1){23}}
\put(3,-2)
{
\put(-17,60){\small $+$}
\put(-17,135){\small $-$}
\put(-17,210){\small $+$}
}
\put(4,-2)
{
\put(30,0){\small $+$}
\put(30,15){\small $-$}
\put(30,30){\small $+$}
\put(30,45){\small $-$}
\put(30,60){\small $+$}
\put(30,75){\small $-$}
\put(30,90){\small $+$}
\put(30,105){\small $-$}
\put(30,120){\small $+$}
\put(30,135){\small $-$}
\put(30,150){\small $+$}
\put(30,165){\small $-$}
\put(30,180){\small $+$}
\put(30,195){\small $-$}
\put(30,210){\small $+$}
\put(30,225){\small $-$}
\put(30,240){\small $+$}
\put(30,255){\small $-$}
\put(30,270){\small $+$}
}
}
%
% fifth diagram
\put(280,0)
{
\put(0,60){\circle{5}}
\put(0,135){\circle{5}}
\put(0,210){\circle{5}}
\put(30,0){\circle*{5}}
\put(30,15){\circle*{5}}
\put(30,30){\circle*{5}}
\put(30,45){\circle*{5}}
\put(30,60){\circle*{5}}
\put(30,75){\circle*{5}}
\put(30,90){\circle*{5}}
\put(30,105){\circle*{5}}
\put(30,120){\circle*{5}}
\put(30,135){\circle*{5}}
\put(30,150){\circle*{5}}
\put(30,165){\circle*{5}}
\put(30,180){\circle*{5}}
\put(30,195){\circle*{5}}
\put(30,210){\circle*{5}}
\put(30,225){\circle*{5}}
\put(30,240){\circle*{5}}
\put(30,255){\circle*{5}}
\put(30,270){\circle*{5}}
%
% vertical arrows
\put(30,3){\vector(0,1){9}}
\put(30,27){\vector(0,-1){9}}
\put(30,33){\vector(0,1){9}}
\put(30,57){\vector(0,-1){9}}
\put(30,63){\vector(0,1){9}}
\put(30,87){\vector(0,-1){9}}
\put(30,93){\vector(0,1){9}}
\put(30,117){\vector(0,-1){9}}
\put(30,123){\vector(0,1){9}}
\put(30,147){\vector(0,-1){9}}
\put(30,153){\vector(0,1){9}}
\put(30,177){\vector(0,-1){9}}
\put(30,183){\vector(0,1){9}}
\put(30,207){\vector(0,-1){9}}
\put(30,213){\vector(0,1){9}}
\put(30,237){\vector(0,-1){9}}
\put(30,243){\vector(0,1){9}}
\put(30,267){\vector(0,-1){9}}
\put(0,132){\vector(0,-1){69}}
\put(0,138){\vector(0,1){69}}
% diagonal arrows
\put(3,60){\vector(1,0){24}}
\put(27,77){\vector(-1,2){23}}
\put(3,128){\vector(2,-3){24}}
\put(27,107){\vector(-1,1){23}}
\put(3,133){\vector(2,-1){24}}
\put(27,135){\vector(-1,0){23}}
\put(3,137){\vector(2,1){24}}
\put(27,163){\vector(-1,-1){23}}
\put(3,142){\vector(2,3){24}}
\put(27,193){\vector(-1,-2){23}}
\put(3,210){\vector(1,0){24}}
\put(3,-2)
{
\put(-17,60){\small $-$}
\put(-17,135){\small $+$}
\put(-17,210){\small $-$}
}
\put(4,-2)
{
\put(30,0){\small $+$}
\put(30,15){\small $-$}
\put(30,30){\small $+$}
\put(30,45){\small $-$}
\put(30,60){\small $+$}
\put(30,75){\small $-$}
\put(30,90){\small $+$}
\put(30,105){\small $-$}
\put(30,120){\small $+$}
\put(30,135){\small $-$}
\put(30,150){\small $+$}
\put(30,165){\small $-$}
\put(30,180){\small $+$}
\put(30,195){\small $-$}
\put(30,210){\small $+$}
\put(30,225){\small $-$}
\put(30,240){\small $+$}
\put(30,255){\small $-$}
\put(30,270){\small $+$}
}
}
\put(10,-20){$Q_1$}
\put(80,-20){$Q_2$}
\put(150,-20){$Q_3$}
\put(220,-20){$Q_4$}
\put(290,-20){$Q_5$}
\end{picture}
\end{center}
%
%%%%%%%%%%%%%%%%%%%%%%%%%%%%%%%%%%%
%
\caption{An example of periodicity of quivers,
where we identify the right columns in all the  quivers
$Q_1$, \dots, $Q_5$.}
\label{fig:quiver}
\end{figure}
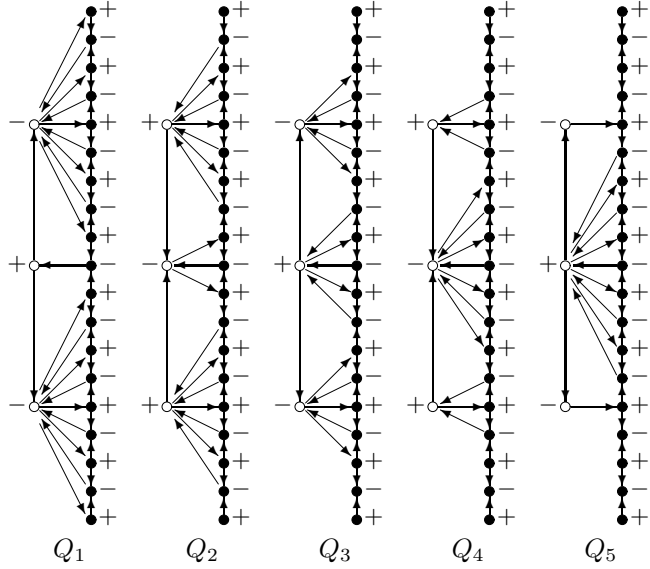

\begin{Example}
\label{ex:G5}
{\em Cluster algebras for
 T- and Y-systems for quantum affinizations
of tamely laced quantum Kac-Moody algebras \cite{Hernandez07a,
Kuniba09,Nakanishi10a}.}

The cluster algebras for
 T- and Y-systems for quantum affinizations
of tamely laced quantum Kac-Moody algebras
 provide a rich family
of more  complicated periodicities of exchange matrices.
As a typical example, let $Q$ be the quiver in Figure \ref{fig:quiver},
where the right columns in the five quivers
$Q_1$, \dots, $Q_5$ are identified.
We remark that $Q$ is not bipartite in the sense of
Example \ref{ex:preperiod}.
Let
$\mathbf{i}^{\bullet}_+$ (resp.  $\mathbf{i}^{\bullet}_-$)
be the sequence of all the distinct elements of
the vertices in $Q$ with property $(\bullet,+)$ (resp. $(\bullet,-)$),
and let
$\mathbf{i}^{\circ}_{+,k}$ (resp.  $\mathbf{i}^{\circ}_{-,k}$)
be a sequence of all the distinct elements of
the vertices in $Q_k$
 with property $(\circ,+)$ (resp. $(\circ,-)$),
where the order of the sequence is chosen arbitrarily
thanks to Lemma \ref{lem:order}.
Let
\begin{align}
\begin{split}
\mathbf{i}&=
\mathbf{i}^{\bullet}_+\,|\,
 \mathbf{i}^{\circ}_{+,1} \,|\,
\mathbf{i}^{\bullet}_-\,|\,
 \mathbf{i}^{\circ}_{+,4} \,|\,
\mathbf{i}^{\bullet}_+\,|\,
 \mathbf{i}^{\circ}_{+,3} \,|\,
\mathbf{i}^{\bullet}_-\,|\,
 \mathbf{i}^{\circ}_{+,2} \,|\,
\mathbf{i}^{\bullet}_+\,|\,
 \mathbf{i}^{\circ}_{+,5} \\
& \quad \,|\,
\mathbf{i}^{\bullet}_-\,|\,
 \mathbf{i}^{\circ}_{-,1} \,|\,
\mathbf{i}^{\bullet}_+\,|\,
 \mathbf{i}^{\circ}_{-,4} \,|\,
\mathbf{i}^{\bullet}_-\,|\,
 \mathbf{i}^{\circ}_{-,3} \,|\,
\mathbf{i}^{\bullet}_+\,|\,
 \mathbf{i}^{\circ}_{-,2} \,|\,
\mathbf{i}^{\bullet}_-\,|\,
 \mathbf{i}^{\circ}_{-,5}.
\end{split}
\end{align}
Let $\sigma$ be 
the permutation of $\{1,\dots,5\}$,
\begin{align}
\sigma= \left(
\begin{matrix}
1&2& 3& 4 & 5\\
3&1& 5& 2 & 4\\
\end{matrix}
\right).
\end{align}
Let $\nu:I \rightarrow I$ be the bijection of order $5$ such that 
each vertex in $Q_i$ maps to the vertex in $Q_{\sigma(i)}$
in the same position. In particular, every vertex with $\bullet$
is a fixed point of $\nu$, and every  vertex with $\circ$
has the $\nu$-orbit of length $5$.
Then
$\mathbf{i}^{\bullet}_+\,|\,
 \mathbf{i}^{\circ}_{+,1} \,|\,
\mathbf{i}^{\bullet}_-\,|\,
 \mathbf{i}^{\circ}_{+,4} 
$ is a $\nu$-period of $Q$,
and $\mathbf{i}$ is a period of $Q$.
Note that $\mathbf{i}\sim_Q
\mathbf{j}(\mathbf{i}^{\bullet}_+\,|\,
 \mathbf{i}^{\circ}_{+,1} \,|\,
\mathbf{i}^{\bullet}_-\,|\,
 \mathbf{i}^{\circ}_{+,4},\nu)$.

The quiver $Q$ corresponds to the level $\ell=4$ T- and Y-systems for
the quantum affinization of the quantum Kac-Moody algebra whose
Cartan matrix is represented by the following Dynkin diagram.
\begin{align*}
\begin{picture}(20,25)(0,-15)
%
% A_r
%
% B_r
\put(0,0){
\put(0,0){\circle{6}}
\put(20,0){\circle{6}}
\drawline(2,-4)(18,-4)
\drawline(2,-2)(18,-2)
\drawline(2,2)(18,2)
\drawline(2,4)(18,4)
\drawline(3,0)(17,0)
\drawline(7,6)(13,0)
\drawline(7,-6)(13,0)
\put(-2,-15){\small $1$}
\put(18,-15){\small $2$}
}
\end{picture}
\end{align*}
See \cite{Nakanishi10a} for more details.
\end{Example}

\begin{Example}
\label{ex:fordy}
{\em Cluster algebras
with $\rho^m$-period of $Q$ for cyclic permutation $\rho$
\cite{Fordy09}.}

For a cyclic permutation $\rho$ of $I$,
many examples of quivers with $\rho^m$-period were constructed
and partially classified in \cite{Fordy09}.
For instance, let $Q$ be the following quiver:

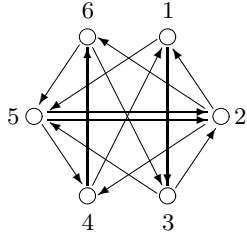
\begin{align*}
\label{eq:fordy}
\raisebox{-45pt}{
\begin{picture}(70,80)(0,-15)
\put(20,0){\circle{6}}
\put(50,0){\circle{6}}
\put(0,30){\circle{6}}
\put(70,30){\circle{6}}
\put(20,60){\circle{6}}
\put(50,60){\circle{6}}
% horizontal
\put(3,27){\vector(2,-3){15}}
\put(17,57){\vector(-2,-3){15}}
\put(53,3){\vector(2,3){15}}
\put(67,33){\vector(-2,3){15}}
\put(5,31.5){\vector(1,0){60}}
\put(5,28.5){\vector(1,0){60}}
\put(46,59){\vector(-3,-2){40}}
\put(46,1){\vector(-3,2){40}}
\put(64,27){\vector(-3,-2){40}}
\put(64,33){\vector(-3,2){40}}
\put(22,4){\vector(1,2){26}}
\put(22,56){\vector(1,-2){26}}
\put(20,4){\vector(0,1){52}}
\put(50,56){\vector(0,-1){52}}
\put(48,67){\small $1$}
\put(75,27){\small $2$}
\put(48,-13){\small $3$}
\put(18,-13){\small $4$}
\put(-10,27){\small $5$}
\put(18,67){\small $6$}
\end{picture}
}
\end{align*}
Let
\begin{align}
\mathbf{i}=(1,2,3,4,5,6).
\end{align}
Let $\rho$ be 
the cyclic permutation of $\{1,\dots,6\}$,
\begin{align}
\rho= \left(
\begin{matrix}
1&2& 3& 4 & 5& 6\\
2&3& 4& 5 & 6& 1\\
\end{matrix}
\right).
\end{align}
Then, $(1,2)$ is a $\rho^2$-period of $Q$,
and $\mathbf{i}$ is a period of $Q$.
Note that $\mathbf{i}=\mathbf{j}((1,2),\rho^2)$.

The quiver $Q$ is a special case of 
a family of `period 2 solutions' of \cite[Section 7.4]{Fordy09}
with $m_1=m_3=-m_2=1$ and $m_{\bar{1}}=0$ therein.
It is also the quiver for the quiver gauge theory
on the del Pezzo 3 surface \cite{Feng01}.
See \cite{Fordy09} for more details.

\end{Example}

\subsection{Examples of periodicities of seeds}

All the examples of periodicities of seeds below
can be proved by verifying the condition \eqref{eq:periody}
in Theorem \ref{thm:tropperiod}, case by case,
with the help of relevant Coxeter elements.

\begin{Example}
\label{ex:finite}
{\em Cluster algebras of finite type\/} \cite{Fomin03a,Fomin03b}.
\par
For a skew symmetrizable matrix $B$,
define a matrix $C=C(B)$ by
\begin{align}
C_{ij}=
\begin{cases}
2 & i=j\\
-|b_{ij}| & i \neq j.\\
\end{cases}
\end{align}
Then, it is known that
the cluster algebra $\mathcal{A}(B,x,y)$ is of finite type
if and only if
$B$ is mutation equivalent to a skew-symmetric matrix $B'$
such that  $C(B')$ is a direct sum of  Cartan matrices of finite type.
Suppose that $C(B)$ is a Cartan matrix of finite type.
Since there are only a finite number of seeds of 
$\mathcal{A}(B,x,y)$,
for any $I$-sequence $\mathbf{i}$,
there is some $p$ such that the $p$-fold concatenation
$\mathbf{i}^p$
of $\mathbf{i}$ is a period of $\mathcal{A}(B,x,y)$.
Among them there is some distinguished period
of $(B,x,y)$,  which is closely
related to the Coxeter element of the Weyl group for $C$.
For type $A_n$, for example, they are given as follows.

(a) Type $A_n$ ($n$: odd).
For odd $n$, let $Q$ be the following alternating quiver with index set
$I=\{1,\dots,n\}$:
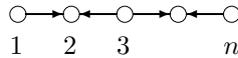
\begin{center}
\begin{picture}(80,25)(0,-15)
%
% A_r
%
% B_r
\put(0,0){\circle{6}}
\put(20,0){\circle{6}}
\put(40,0){\circle{6}}
\put(60,0){\circle{6}}
\put(80,0){\circle{6}}
\put(3,0){\vector(1,0){14}}
\put(37,0){\vector(-1,0){14}}
\put(43,0){\vector(1,0){14}}
\put(77,0){\vector(-1,0){14}}
\put(-3,-15){$1$}
\put(17,-15){$2$}
\put(37,-15){$3$}
\put(77,-15){$n$}
\end{picture}
\end{center}
Let $\mathbf{i}=\mathbf{i}_+\,|\, \mathbf{i}_-$,
where
$\mathbf{i}_+=(1,3,\dots,n)$,  $\mathbf{i}_-=(2,4,\dots,n-1)$.
Let $\omega:I \rightarrow I$ be the 
left-right reflection,
which is a quiver automorphism of $Q$.
Then $\mathbf{i}$ is a period of $Q$.
Furthermore, 
$\mathbf{i}^{(n+3)/2}$ is an $\omega$-period of $(Q,x,y)$,
and 
$\mathbf{i}^{n+3}$ is a period of $(Q,x,y)$.
We note that $n+3=h(A_{n})+2$,
where $h(X)$ is the
 {\em  Coxeter number\/} of type $X$.
For simply laced $X$, $h(X)$ coincides with the
{\em dual Coxeter number\/} $h^{\vee}(X)$ of type $X$.

(b) Type $A_n$ ($n$: even).
For even $n$, let $Q$ be the following quiver with index set
$I=\{1,\dots,n\}$:
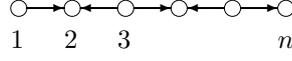
\begin{center}
\begin{picture}(100,25)(0,-15)
%
% A_r
%
% B_r
\put(0,0){\circle{6}}
\put(20,0){\circle{6}}
\put(40,0){\circle{6}}
\put(60,0){\circle{6}}
\put(80,0){\circle{6}}
\put(100,0){\circle{6}}
\put(3,0){\vector(1,0){14}}
\put(37,0){\vector(-1,0){14}}
\put(43,0){\vector(1,0){14}}
\put(77,0){\vector(-1,0){14}}
\put(83,0){\vector(1,0){14}}
\put(-3,-15){$1$}
\put(17,-15){$2$}
\put(37,-15){$3$}
\put(97,-15){$n$}
\end{picture}
\end{center}
Let $\mathbf{i}=\mathbf{i}_+\,|\, \mathbf{i}_-$,
where
$\mathbf{i}_+=(1,3,\dots,n-1)$,  $\mathbf{i}_-=(2,4,\dots,n)$.
Let $\nu:I \rightarrow I$ be the left-right
reflection,
which is {\em not\/} a quiver automorphism of $Q$.
Then $\mathbf{i}_+$ is an $\nu$-period of $Q$
and $\mathbf{i}$ is a period of $Q$.
Furthermore, 
$\mathbf{i}^{n/2+1}\,|\, \mathbf{i}_+$ is a $\nu$-period of $(Q,x,y)$,
and 
$\mathbf{i}^{n+3}$ is a period of $(Q,x,y)$.
Note that $\mathbf{i}\sim_Q \mathbf{j}(\mathbf{i}_+,\nu)$
and $\mathbf{i}^{n+3}\sim_Q\mathbf{j}(\mathbf{i}^{n/2+1}\mathbf{i}_+,\nu)$.
\end{Example}

\begin{Example}
\label{ex:affine}
{\em Cluster algebras for T- and Y-systems of
 quantum affine algebras\/}
 \cite{Keller08,Keller10,DiFrancesco09a,Inoue10a,Inoue10b,Inoue10c}.

With each  pair $(X,\ell)$ of a Dynkin diagram $X$ of finite type
and an integer $\ell \geq 2$, 
one can associate a quiver $Q=Q(X,\ell)$.
They are related to the T- and Y-systems of
a quantum affine algebra of type $X$,
and  provide a family of periodicities of seeds.
Let us give typical examples
for simply laced and nonsimply laced ones.

(a) {\em Simply laced case:} $(X,\ell)=(A_4,4)$.
Let $Q$ be the following quiver with index set $I$:
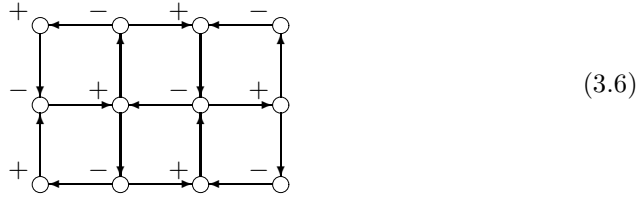
\begin{align}
\label{eq:qA}
\raisebox{-40pt}{
\begin{picture}(90,80)(0,-5)
\put(0,0){\circle{6}}
\put(30,0){\circle{6}}
\put(60,0){\circle{6}}
\put(90,0){\circle{6}}
\put(0,30){\circle{6}}
\put(30,30){\circle{6}}
\put(60,30){\circle{6}}
\put(90,30){\circle{6}}
\put(0,60){\circle{6}}
\put(30,60){\circle{6}}
\put(60,60){\circle{6}}
\put(90,60){\circle{6}}
% vertical
\put(0,3){\vector(0,1){24}}
\put(0,57){\vector(0,-1){24}}
\put(30,27){\vector(0,-1){24}}
\put(30,33){\vector(0,1){24}}
\put(60,3){\vector(0,1){24}}
\put(60,57){\vector(0,-1){24}}
\put(90,27){\vector(0,-1){24}}
\put(90,33){\vector(0,1){24}}
% horizontal
\put(27,0){\vector(-1,0){24}}
\put(33,0){\vector(1,0){24}}
\put(87,0){\vector(-1,0){24}}
\put(3,30){\vector(1,0){24}}
\put(57,30){\vector(-1,0){24}}
\put(63,30){\vector(1,0){24}}
\put(27,60){\vector(-1,0){24}}
\put(33,60){\vector(1,0){24}}
\put(87,60){\vector(-1,0){24}}
\put(-12,3){$+$}
\put(18,3){$-$}
\put(48,3){$+$}
\put(78,3){$-$}
\put(-12,33){$-$}
\put(18,33){$+$}
\put(48,33){$-$}
\put(78,33){$+$}
\put(-12,63){$+$}
\put(18,63){$-$}
\put(48,63){$+$}
\put(78,63){$-$}
\end{picture}
}
\end{align}

\noindent
Let
$\mathbf{i}_+$  and  $\mathbf{i}_-$ be as before.
Let $\mathbf{i}=\mathbf{i}_+\,|\, \mathbf{i}_-$.
Let $\nu:I \rightarrow I$
be  the left-right reflection,
and let $\omega :I \rightarrow I$
be the top-bottom reflection,
so that $\omega\nu=\nu\omega$.
Then $\mathbf{i}_+$ is a $\nu$-period of $Q$,
and $\mathbf{i}$ is a period of $Q$.
Furthermore, 
$\mathbf{i}^{4}\,|\, \mathbf{i}_+$
 is a $\nu\omega$-period of $(Q,x,y)$,
and 
$\mathbf{i}^{9}$ is a period of $(Q,x,y)$,
where $9=5+4=h(A_4)+\ell$.
Note that $\mathbf{i}\sim_Q \mathbf{j}(\mathbf{i}_+,\nu)$
and $\mathbf{i}^{9}\sim_Q\mathbf{j}(\mathbf{i}^{4}\mathbf{i}_+,\nu\omega)$.

(b) {\em Nonsimply laced case:} $(X,\ell)=(B_4,4)$.
Let $Q$ be the following quiver with index set $I$:
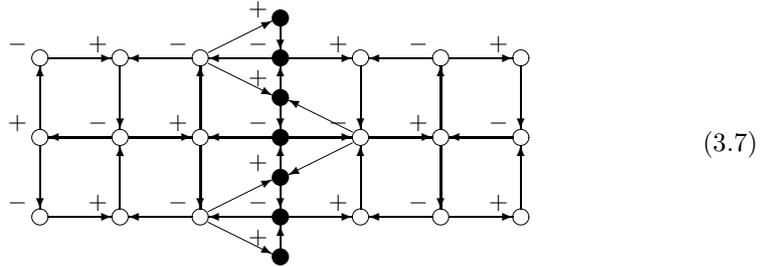
\begin{align}
\label{eq:qB}
\setlength{\unitlength}{1pt}
\raisebox{-45pt}{
\begin{picture}(180,105)(0,-20)
\put(0,0){\circle{6}}
\put(30,0){\circle{6}}
\put(60,0){\circle{6}}
\put(90,0){\circle*{6}}
\put(120,0){\circle{6}}
\put(150,0){\circle{6}}
\put(180,0){\circle{6}}
\put(90,15){\circle*{6}}
\put(3,0){\vector(1,0){24}}
\put(57,0){\vector(-1,0){24}}
\put(87,0){\vector(-1,0){24}}
\put(93,0){\vector(1,0){24}}
\put(147,0){\vector(-1,0){24}}
\put(153,0){\vector(1,0){24}}
\put(0,30)
{
\put(0,0){\circle{6}}
\put(30,0){\circle{6}}
\put(60,0){\circle{6}}
\put(90,0){\circle*{6}}
\put(120,0){\circle{6}}
\put(150,0){\circle{6}}
\put(180,0){\circle{6}}
\put(90,15){\circle*{6}}
\put(27,0){\vector(-1,0){24}}
\put(33,0){\vector(1,0){24}}
\put(87,0){\vector(-1,0){24}}
\put(93,0){\vector(1,0){24}}
\put(123,0){\vector(1,0){24}}
\put(177,0){\vector(-1,0){24}}
}
\put(0,60)
{
\put(0,0){\circle{6}}
\put(30,0){\circle{6}}
\put(60,0){\circle{6}}
\put(90,0){\circle*{6}}
\put(120,0){\circle{6}}
\put(150,0){\circle{6}}
\put(180,0){\circle{6}}
\put(90,15){\circle*{6}}
\put(3,0){\vector(1,0){24}}
\put(57,0){\vector(-1,0){24}}
\put(87,0){\vector(-1,0){24}}
\put(93,0){\vector(1,0){24}}
\put(147,0){\vector(-1,0){24}}
\put(153,0){\vector(1,0){24}}
}
\put(90,-15){\circle*{6}}
%
% vertical arrows
\put(90,-12){\vector(0,1){9}}
\put(90,12){\vector(0,-1){9}}
\put(90,18){\vector(0,1){9}}
\put(90,42){\vector(0,-1){9}}
\put(90,48){\vector(0,1){9}}
\put(90,72){\vector(0,-1){9}}
\put(63,-2){\vector(2,-1){24}}
\put(63,2){\vector(2,1){24}}
\put(63,58){\vector(2,-1){24}}
\put(63,62){\vector(2,1){24}}
\put(117,28){\vector(-2,-1){24}}
\put(117,32){\vector(-2,1){24}}
%
%\put(-30,3){\vector(0,1){24}}
\put(0,27){\vector(0,-1){24}}
\put(30,3){\vector(0,1){24}}
\put(60,27){\vector(0,-1){24}}
\put(120,3){\vector(0,1){24}}
\put(150,27){\vector(0,-1){24}}
\put(180,3){\vector(0,1){24}}
%\put(210,27){\vector(0,-1){24}}
%
%\put(-30,57){\vector(0,-1){24}}
\put(0,33){\vector(0,1){24}}
\put(30,57){\vector(0,-1){24}}
\put(60,33){\vector(0,1){24}}
\put(120,57){\vector(0,-1){24}}
\put(150,33){\vector(0,1){24}}
\put(180,57){\vector(0,-1){24}}
%\put(210,33){\vector(0,1){24}}
%
\put(-14,1)
{
\put(2,2){\small $-$}
\put(2,32){\small $+$}
\put(2,62){\small $-$}
\put(32,2){\small $+$}
\put(32,32){\small $-$}
\put(32,62){\small $+$}
\put(62,2){\small $-$}
\put(62,32){\small $+$}
\put(62,62){\small $-$}
\put(92,-11){\small $+$}
\put(92,2){\small $-$}
\put(92,17){\small $+$}
\put(92,32){\small $-$}
\put(92,49){\small $+$}
\put(92,62){\small $-$}
\put(92,75){\small $+$}
\put(122,2){\small $+$}
\put(122,32){\small $-$}
\put(122,62){\small $+$}
\put(152,2){\small $-$}
\put(152,32){\small $+$}
\put(152,62){\small $-$}
\put(182,2){\small $+$}
\put(182,32){\small $-$}
\put(182,62){\small $+$}
}
\end{picture}
}
\end{align}

\noindent
Let
$\mathbf{i}^{\bullet}_+$ (resp.  $\mathbf{i}^{\bullet}_-$,
$\mathbf{i}^{\circ}_+$, $\mathbf{i}^{\circ}_-$)
be a sequence of all the distinct elements of
$I$ with property $(\bullet,+)$ (resp. $(\bullet,-)$, $(\circ,+)$,
$(\circ,-)$),
where the order of the sequence is chosen arbitrarily.
Let
\begin{align}
\label{eq:slice22}
\mathbf{i}=
(\mathbf{i}^{\bullet}_+\,|\,
 \mathbf{i}^{\circ}_+ \,|\,
\mathbf{i}^{\bullet}_-)\,|\,
(\mathbf{i}^{\bullet}_+\,|\,
 \mathbf{i}^{\circ}_- \,|\,
\mathbf{i}^{\bullet}_-).
\end{align}
Let $\nu:I \rightarrow I$
be  the left-right reflection,
and let $\omega :I \rightarrow I$
be the top-bottom reflection,
so that $\omega\nu=\nu\omega$.
Then $\mathbf{i}^{\bullet}_+\,|\,
 \mathbf{i}^{\circ}_+ \,|\,
\mathbf{i}^{\bullet}_-
$ is a $\nu$-period of $Q$,
and $\mathbf{i}$ is a period of $Q$.
Furthermore, 
$\mathbf{i}^{5}\,|\, 
(\mathbf{i}^{\bullet}_+\,|\,
 \mathbf{i}^{\circ}_+ \,|\,
\mathbf{i}^{\bullet}_-)$
 is a $\nu\omega$-period of $(Q,x,y)$,
and 
$\mathbf{i}^{11}$ is a period of $(Q,x,y)$,
where $11=7+4=h^{\vee}(B_4)+\ell$.
Note that $\mathbf{i}\sim_Q \mathbf{j}(\mathbf{i}^{\bullet}_+
\, | \, \mathbf{i}^{\circ}_+\, | \,
\mathbf{i}^{\bullet}_-,\nu)$
and $\mathbf{i}^{11}\sim_Q\mathbf{j}(
\mathbf{i}^{5}\,|\, 
(\mathbf{i}^{\bullet}_+\,|\,
 \mathbf{i}^{\circ}_+ \,|\,
\mathbf{i}^{\bullet}_-),\nu\omega)$.
See \cite{Inoue10a} for more details.
\end{Example}

\begin{Example}
{\em Cluster algebras for sine-Gordon T- and Y-systems\/}
 \cite{Nakanishi10b}.
\label{ex:sine}

The T- and Y-systems which originated from
the sine-Gordon model provide another family of
periodicities of seeds.
Let us give an example.

Let $Q$ be the following quiver with index set $I=\{1,
\dots, 13\}$.
Here all the  vertices with $\bullet$ in the same position
 in  the quivers $Q_1$,\dots,$Q_{6}$ are
identified. The quiver $Q$ is mutation equivalent to
the quiver of type $D_{13}$.

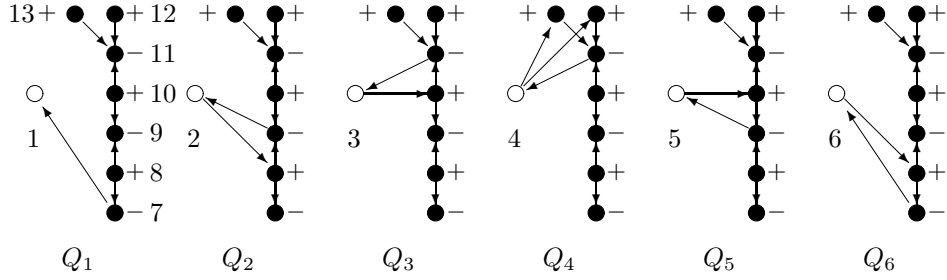
\begin{center}
\setlength{\unitlength}{1pt}
%
%%%%%%%%%%%%%%%%%%%%%%%%%%%%%%%%%%%
%
\begin{picture}(330,110)(0,-25)
%
% first diagram
\put(0,0)
{ 
\put(0,45){\circle{6}}
\put(30,0){\circle*{6}}
\put(30,15){\circle*{6}}
\put(30,30){\circle*{6}}
\put(30,45){\circle*{6}}
\put(30,60){\circle*{6}}
\put(15,75){\circle*{6}}
\put(30,75){\circle*{6}}
%
% vertical arrows
\put(30,12){\vector(0,-1){9}}
\put(30,18){\vector(0,1){9}}
\put(30,42){\vector(0,-1){9}}
\put(30,48){\vector(0,1){9}}
\put(30,72){\vector(0,-1){9}}
\put(18,72){\vector(1,-1){9}}
%\put(30,3){\vector(0,1){9}}
%\put(30,27){\vector(0,-1){9}}
%\put(30,33){\vector(0,1){9}}
%\put(30,57){\vector(0,-1){9}}
%\put(30,63){\vector(0,1){9}}
% diagonal arrows
\put(27,4){\vector(-2,3){24}}
% diagonal arrows
%
%\put(27,77){\vector(-1,2){23}}
%\put(3,128){\vector(2,-3){24}}
%\put(27,107){\vector(-1,1){23}}
%\put(3,133){\vector(2,-1){24}}
%\put(27,135){\vector(-1,0){23}}
%\put(3,137){\vector(2,1){24}}
%\put(27,163){\vector(-1,-1){23}}
%\put(3,142){\vector(2,3){24}}
%\put(27,193){\vector(-1,-2){23}}
%\put(3,51){\vector(1,-2){24}}
%\put(27,19){\vector(-2,3){23}}
%\put(3,58){\vector(1,-1){24}}
%\put(27,47){\vector(-2,1){23}}
%\put(3,60){\vector(1,0){24}}
%\put(27,73){\vector(-2,-1){23}}
%\put(3,62){\vector(1,1){24}}
%\put(27,101){\vector(-2,-3){23}}
%
%
\put(3,-2)
{
%\put(-17,60){\small $+$}
}
\put(4,-2)
{
\put(30,0){\small $-$}
\put(30,15){\small $+$}
\put(30,30){\small $-$}
\put(30,45){\small $+$}
\put(30,60){\small $-$}
\put(30,75){\small $+$}
\put(-3,75){\small $+$}
}
}
% second diagram
\put(60,0)
{
\put(0,45){\circle{6}}
\put(30,0){\circle*{6}}
\put(30,15){\circle*{6}}
\put(30,30){\circle*{6}}
\put(30,45){\circle*{6}}
\put(30,60){\circle*{6}}
\put(15,75){\circle*{6}}
\put(30,75){\circle*{6}}
%
% vertical arrows
\put(30,12){\vector(0,-1){9}}
\put(30,18){\vector(0,1){9}}
\put(30,42){\vector(0,-1){9}}
\put(30,48){\vector(0,1){9}}
\put(30,72){\vector(0,-1){9}}
\put(18,72){\vector(1,-1){9}}
%\put(30,3){\vector(0,1){9}}
%\put(30,27){\vector(0,-1){9}}
%\put(30,33){\vector(0,1){9}}
%\put(30,57){\vector(0,-1){9}}
%\put(30,63){\vector(0,1){9}}
% diagonal arrows
\put(3,43){\vector(1,-1){24}}
\put(27,32){\vector(-2,1){23}}
% diagonal arrows
%
%\put(27,77){\vector(-1,2){23}}
%\put(3,128){\vector(2,-3){24}}
%\put(27,107){\vector(-1,1){23}}
%\put(3,133){\vector(2,-1){24}}
%\put(27,135){\vector(-1,0){23}}
%\put(3,137){\vector(2,1){24}}
%\put(27,163){\vector(-1,-1){23}}
%\put(3,142){\vector(2,3){24}}
%\put(27,193){\vector(-1,-2){23}}
%\put(3,51){\vector(1,-2){24}}
%\put(27,19){\vector(-2,3){23}}
%\put(3,58){\vector(1,-1){24}}
%\put(27,47){\vector(-2,1){23}}
%\put(3,60){\vector(1,0){24}}
%\put(27,73){\vector(-2,-1){23}}
%\put(3,62){\vector(1,1){24}}
%\put(27,101){\vector(-2,-3){23}}
%
%
\put(3,-2)
{
%\put(-17,60){\small $+$}
}
\put(4,-2)
{
\put(30,0){\small $-$}
\put(30,15){\small $+$}
\put(30,30){\small $-$}
\put(30,45){\small $+$}
\put(30,60){\small $-$}
\put(30,75){\small $+$}
\put(-3,75){\small $+$}
}
}
% third diagram
\put(120,0)
{
\put(0,45){\circle{6}}
\put(30,0){\circle*{6}}
\put(30,15){\circle*{6}}
\put(30,30){\circle*{6}}
\put(30,45){\circle*{6}}
\put(30,60){\circle*{6}}
\put(15,75){\circle*{6}}
\put(30,75){\circle*{6}}
%
% vertical arrows
\put(30,12){\vector(0,-1){9}}
\put(30,18){\vector(0,1){9}}
\put(30,42){\vector(0,-1){9}}
\put(30,48){\vector(0,1){9}}
\put(30,72){\vector(0,-1){9}}
\put(18,72){\vector(1,-1){9}}
%\put(30,3){\vector(0,1){9}}
%\put(30,27){\vector(0,-1){9}}
%\put(30,33){\vector(0,1){9}}
%\put(30,57){\vector(0,-1){9}}
%\put(30,63){\vector(0,1){9}}
% diagonal arrows
\put(3,45){\vector(1,0){24}}
\put(27,58){\vector(-2,-1){23}}
% diagonal arrows
%
%\put(27,77){\vector(-1,2){23}}
%\put(3,128){\vector(2,-3){24}}
%\put(27,107){\vector(-1,1){23}}
%\put(3,133){\vector(2,-1){24}}
%\put(27,135){\vector(-1,0){23}}
%\put(3,137){\vector(2,1){24}}
%\put(27,163){\vector(-1,-1){23}}
%\put(3,142){\vector(2,3){24}}
%\put(27,193){\vector(-1,-2){23}}
%\put(3,51){\vector(1,-2){24}}
%\put(27,19){\vector(-2,3){23}}
%\put(3,58){\vector(1,-1){24}}
%\put(27,47){\vector(-2,1){23}}
%\put(3,60){\vector(1,0){24}}
%\put(27,73){\vector(-2,-1){23}}
%\put(3,62){\vector(1,1){24}}
%\put(27,101){\vector(-2,-3){23}}
%
%
\put(3,-2)
{
%\put(-17,60){\small $+$}
}
\put(4,-2)
{
\put(30,0){\small $-$}
\put(30,15){\small $+$}
\put(30,30){\small $-$}
\put(30,45){\small $+$}
\put(30,60){\small $-$}
\put(30,75){\small $+$}
\put(-3,75){\small $+$}
}
}
% 4th diagram
\put(180,0)
{
\put(0,45){\circle{6}}
\put(30,0){\circle*{6}}
\put(30,15){\circle*{6}}
\put(30,30){\circle*{6}}
\put(30,45){\circle*{6}}
\put(30,60){\circle*{6}}
\put(15,75){\circle*{6}}
\put(30,75){\circle*{6}}
%
% vertical arrows
\put(30,12){\vector(0,-1){9}}
\put(30,18){\vector(0,1){9}}
\put(30,42){\vector(0,-1){9}}
\put(30,48){\vector(0,1){9}}
\put(30,72){\vector(0,-1){9}}
\put(18,72){\vector(1,-1){9}}
%\put(30,3){\vector(0,1){9}}
%\put(30,27){\vector(0,-1){9}}
%\put(30,33){\vector(0,1){9}}
%\put(30,57){\vector(0,-1){9}}
%\put(30,63){\vector(0,1){9}}
% diagonal arrows
\put(3,48){\vector(1,1){25}}
\put(27,58){\vector(-2,-1){23}}
\put(2,48){\vector(1,2){11}}
% diagonal arrows
%
%\put(27,77){\vector(-1,2){23}}
%\put(3,128){\vector(2,-3){24}}
%\put(27,107){\vector(-1,1){23}}
%\put(3,133){\vector(2,-1){24}}
%\put(27,135){\vector(-1,0){23}}
%\put(3,137){\vector(2,1){24}}
%\put(27,163){\vector(-1,-1){23}}
%\put(3,142){\vector(2,3){24}}
%\put(27,193){\vector(-1,-2){23}}
%\put(3,51){\vector(1,-2){24}}
%\put(27,19){\vector(-2,3){23}}
%\put(3,58){\vector(1,-1){24}}
%\put(27,47){\vector(-2,1){23}}
%\put(3,60){\vector(1,0){24}}
%\put(27,73){\vector(-2,-1){23}}
%\put(3,62){\vector(1,1){24}}
%\put(27,101){\vector(-2,-3){23}}
%
%
\put(3,-2)
{
%\put(-17,60){\small $+$}
}
\put(4,-2)
{
\put(30,0){\small $-$}
\put(30,15){\small $+$}
\put(30,30){\small $-$}
\put(30,45){\small $+$}
\put(30,60){\small $-$}
\put(30,75){\small $+$}
\put(-3,75){\small $+$}
}
}
% 5thsecond diagram
\put(240,0)
{
\put(0,45){\circle{6}}
\put(30,0){\circle*{6}}
\put(30,15){\circle*{6}}
\put(30,30){\circle*{6}}
\put(30,45){\circle*{6}}
\put(30,60){\circle*{6}}
\put(15,75){\circle*{6}}
\put(30,75){\circle*{6}}
%
% vertical arrows
\put(30,12){\vector(0,-1){9}}
\put(30,18){\vector(0,1){9}}
\put(30,42){\vector(0,-1){9}}
\put(30,48){\vector(0,1){9}}
\put(30,72){\vector(0,-1){9}}
\put(18,72){\vector(1,-1){9}}
%\put(30,3){\vector(0,1){9}}
%\put(30,27){\vector(0,-1){9}}
%\put(30,33){\vector(0,1){9}}
%\put(30,57){\vector(0,-1){9}}
%\put(30,63){\vector(0,1){9}}
% diagonal arrows
\put(3,45){\vector(1,0){24}}
\put(27,32){\vector(-2,1){23}}
% diagonal arrows
%
%\put(27,77){\vector(-1,2){23}}
%\put(3,128){\vector(2,-3){24}}
%\put(27,107){\vector(-1,1){23}}
%\put(3,133){\vector(2,-1){24}}
%\put(27,135){\vector(-1,0){23}}
%\put(3,137){\vector(2,1){24}}
%\put(27,163){\vector(-1,-1){23}}
%\put(3,142){\vector(2,3){24}}
%\put(27,193){\vector(-1,-2){23}}
%\put(3,51){\vector(1,-2){24}}
%\put(27,19){\vector(-2,3){23}}
%\put(3,58){\vector(1,-1){24}}
%\put(27,47){\vector(-2,1){23}}
%\put(3,60){\vector(1,0){24}}
%\put(27,73){\vector(-2,-1){23}}
%\put(3,62){\vector(1,1){24}}
%\put(27,101){\vector(-2,-3){23}}
%
%
\put(3,-2)
{
%\put(-17,60){\small $+$}
}
\put(4,-2)
{
\put(30,0){\small $-$}
\put(30,15){\small $+$}
\put(30,30){\small $-$}
\put(30,45){\small $+$}
\put(30,60){\small $-$}
\put(30,75){\small $+$}
\put(-3,75){\small $+$}
}
}
% 6th diagram
\put(300,0)
{
\put(0,45){\circle{6}}
\put(30,0){\circle*{6}}
\put(30,15){\circle*{6}}
\put(30,30){\circle*{6}}
\put(30,45){\circle*{6}}
\put(30,60){\circle*{6}}
\put(15,75){\circle*{6}}
\put(30,75){\circle*{6}}
%
% vertical arrows
\put(30,12){\vector(0,-1){9}}
\put(30,18){\vector(0,1){9}}
\put(30,42){\vector(0,-1){9}}
\put(30,48){\vector(0,1){9}}
\put(30,72){\vector(0,-1){9}}
\put(18,72){\vector(1,-1){9}}
%\put(30,3){\vector(0,1){9}}
%\put(30,27){\vector(0,-1){9}}
%\put(30,33){\vector(0,1){9}}
%\put(30,57){\vector(0,-1){9}}
%\put(30,63){\vector(0,1){9}}
% diagonal arrows
\put(27,4){\vector(-2,3){23}}
\put(3,43){\vector(1,-1){24}}
% diagonal arrows
%
%\put(27,77){\vector(-1,2){23}}
%\put(3,128){\vector(2,-3){24}}
%\put(27,107){\vector(-1,1){23}}
%\put(3,133){\vector(2,-1){24}}
%\put(27,135){\vector(-1,0){23}}
%\put(3,137){\vector(2,1){24}}
%\put(27,163){\vector(-1,-1){23}}
%\put(3,142){\vector(2,3){24}}
%\put(27,193){\vector(-1,-2){23}}
%\put(3,51){\vector(1,-2){24}}
%\put(27,19){\vector(-2,3){23}}
%\put(3,58){\vector(1,-1){24}}
%\put(27,47){\vector(-2,1){23}}
%\put(3,60){\vector(1,0){24}}
%\put(27,73){\vector(-2,-1){23}}
%\put(3,62){\vector(1,1){24}}
%\put(27,101){\vector(-2,-3){23}}
%
%
\put(3,-2)
{
%\put(-17,60){\small $+$}
}
\put(4,-2)
{
\put(30,0){\small $-$}
\put(30,15){\small $+$}
\put(30,30){\small $-$}
\put(30,45){\small $+$}
\put(30,60){\small $-$}
\put(30,75){\small $+$}
\put(-3,75){\small $+$}
}
}
\put(10,-20){$Q_1$}
\put(70,-20){$Q_2$}
\put(130,-20){$Q_3$}
\put(190,-20){$Q_4$}
\put(250,-20){$Q_5$}
\put(310,-20){$Q_6$}
\put(-3,25){$1$}
\put(57,25){$2$}
\put(117,25){$3$}
\put(177,25){$4$}
\put(237,25){$5$}
\put(297,25){$6$}
\put(43,-3){$7$}
\put(43,12){$8$}
\put(43,27){$9$}
\put(43,42){$10$}
\put(43,57){$11$}
\put(43,72){$12$}
\put(-10,72){$13$}
%\put(355,170.5){${\scriptstyle\ell -1}\left\{ \makebox(0,115){}\right.$}
%\put(330,170.5){$\left. \makebox(0,175){}\right\}{\scriptstyle t\ell -1}$}
\end{picture}
\end{center}
\noindent
Let
$\mathbf{i}^{\bullet}_+$ and  $\mathbf{i}^{\bullet}_-$
be  the $I$-sequences as before.
Let
\begin{align}
\mathbf{i}=
(\mathbf{i}^{\bullet}_+\,|\,
 (1) \,|\,
\mathbf{i}^{\bullet}_-
)\,|\,
(\mathbf{i}^{\bullet}_+\,|\,
 (2) \,|\,
\mathbf{i}^{\bullet}_-
)\,|\,
\cdots
\,|\,
(\mathbf{i}^{\bullet}_+\,|\,
 (6) \,|\,
\mathbf{i}^{\bullet}_-
).
\end{align}
Let $\nu:I \rightarrow I$
be the bijection of order $6$  cyclically mapping
the vertices 1, 2, \dots, 6 to 2, 3, \dots, 1
 and fixing the rest.
Let $\omega:I \rightarrow I$
be the involution  exchanging the top two vertices with $\bullet$ 
and fixing the rest,
which is a quiver automorphism of $Q$.
Note that $\omega\nu=\nu\omega$.
Then $\mathbf{i}^{\bullet}_+\,|\,
 (1) \,|\,
\mathbf{i}^{\bullet}_-$ is a $\nu$-period of $Q$,
and
$\mathbf{i}$ is a period of $Q$.
Furthermore, 
$\mathbf{i}^{2}\,|\, (\mathbf{i}^{\bullet}_+\,|\,
 (1) \,|\,
\mathbf{i}^{\bullet}_-)$
 is a $\nu\omega$-period of $(Q,x,y)$,
and 
$\mathbf{i}^{13}$ is a period of $(Q,x,y)$,
where $13=(12+2+10+2)/2 = (h(D_7)+2+h(D_6)+2)/2$.
Note that $\mathbf{i}\sim_Q \mathbf{j}(\mathbf{i}^{\bullet}_+
\, | \, ( 1) \, | \,
\mathbf{i}^{\bullet}_-,\nu)$
and $\mathbf{i}^{13}\sim_Q\mathbf{j}(
\mathbf{i}^{2}\,|\, 
(\mathbf{i}^{\bullet}_+\,|\,
 ( 1) \,|\,
\mathbf{i}^{\bullet}_-),\nu\omega)$.
See \cite{Nakanishi10b} for more details.
\end{Example}

In summary,
we make two plain observations in these examples.

\begin{itemize}
\item[(a)] There are a variety of patterns of periodicities
of seeds.

\item[(b)] It is hard to tell whether a
repetition of a given
period of exchange matrices yields
a period of seeds
 just by looking at the shape of 
its quiver.
\end{itemize}

We expect that these examples are  just a tip of iceberg
of the whole class of periodicities of seeds
and their classification would be very challenging
but interesting.

\section{Restriction/Extension Theorem}

In this section we present Restriction/Extension Theorem
on periodicities of seeds.

We first state a very general theorem
on the relation between $g$-vectors and tropical coefficients.

Let $I=\{1,\dots,n\}$.
For a given seed $(B',x',y')$ of $\mathcal{A}(B,x,y)$,
let $\mathbf{g}'_i=(g'_{1i},\dots,g'_{ni})$ ($i\in I$)
be the $g$-vectors defined by \eqref{eq:gF}.
We also introduce integers $c'_{ij}$ ($i,j=1,\dots,n$)
by
\begin{align}
[y'_i]_{\mathbf{T}}=
\prod_{j=1}^n y_j^{c'_{ji}}.
\end{align}
Consider the matrices $C'=(c'_{ij})_{i,j=1}^n$
and $G'=(g'_{ij})_{i,j=1}^n$;
that is,  each column of $C'$ is the exponents
of a tropical coefficient,
and each column of $G'$ is a $g$-vector.
We remark that $C'$ is the bottom part of the extended exchange matrix
$\tilde{B}'$ 
with the principal coefficients
in \cite[Eq.~(2.14)]{Fomin07}.

\begin{Theorem}
\label{thm:CG}
Assume that $B$ is skew symmetric.
Then, the matrices ${C'}^{T}$ and $G'$ are inverse to each other,
where ${C'}^{T}$ is the transpose of $C'$.
\end{Theorem}

\begin{proof}
We prove it by the induction on mutations.
When $(B',x',y')=(B,x,y)$, the claim holds because
$C'=G'=I$.
Suppose that $(B'',x'',y'')=\mu_k(B',x',y')$.
Let $C''$ and $G''$ be the corresponding matrices.
Then, the following recursion relations hold
 \cite[Eq.(5.9) \& Proposition 6.6]{Fomin07}
\begin{align}
\label{eq:cmut}
c''_{ij}&=
\begin{cases}
-c'_{ik}& j=k\\
c'_{ij} + \frac{1}{2}(|c'_{ik}|b'_{kj}+ c'_{ik}|b'_{kj}|)
& j \neq k,
\end{cases}
\\
\label{eq:gmut}
g''_{ij}
&=
\begin{cases}
\displaystyle
-g'_{ik} + 
\sum_{p=1}^n g'_{ip} [b'_{pk}]_+ 
-
\sum_{p=1}^n b_{ip} [c'_{pk}]_+ 
&
j=k,\\
g'_{ij}& j\neq k,
\end{cases}
\end{align}
where $[a]_+=a$ if $a> 0$ and 0 otherwise.
Also the following relation holds \cite[Eq.(6.14)]{Fomin07}
\begin{align}
\sum_{p=1}^n g'_{ip} b'_{pj}
=\sum_{p=1}^n b_{ip} c'_{pj}.
\end{align}
%So far, we did not use the assumption that
%$B$ is skew symmetric.
%However,
We need one more fact: for each $j$, $c_{ij}$ 
are simultaneously nonnegative or nonpositive
for all $j=1,\dots,n$.
% (Conjecture
%\ref{conj:pos} (a)).
This is proved for $B$ is skew symmetric
by Theorem \ref{thm:pos}.
Then, using these facts and the induction hypothesis ${C'}^{T}G'=I$,
the claim ${C''}^{T}G''=I$ can be easily verified.
\end{proof}

\begin{Remark}
Theorem \ref{thm:CG} appeared in \cite{Keller10,Plamondon10b} implicitly.
Namely, the claim is also a consequence of
the following formula
in \cite[Corollaries 6.9 \& 6.13]{Keller10}
and  \cite[Proposition 3.6]{Plamondon10b}:
\begin{align}
g'_{ij}&=[\mathrm{ind}_{T} T'_j:T_i] ,\\
c'_{ij}&=[\mathrm{ind}^{\mathrm{op}}_{T'}T_i:T'_j] ,
\end{align}
where $T=\bigoplus_{i=1}^n T_i$ 
and $T'=\bigoplus_{i=1}^n T'_i$ are
some objects in the `cluster category' for $\mathcal{A}(B,x,y)$.
We ask the reader to consult \cite{Plamondon10b} and
the forthcoming paper \cite{Plamondon10c} for details.
\end{Remark}

Now let us turn to the main statement of the section.

For a pair of index sets $I\subset \tilde{I}$,
suppose that
there is a pair of skew symmetric matrices
 $B
=(b_{ij})_{i,j\in I}$ 
and
 $\tilde{B}
=(\tilde{b}_{ij})_{i,j\in \tilde{I}}$ 
such that
$B=\tilde{B}\vert_I$ under the restriction of the index set
$\tilde{I}$ to $I$.
 (In terms of quivers,
 $Q$ is a full subquiver of $\tilde{Q}$.)
Then, we say that
$B$ is the {\em $I$-restriction\/} of $\tilde{B}$
and
 $\tilde{B}$ is an {\em $\tilde{I}$-extension\/}
of $B$.

\begin{Theorem}[Restriction/Extension Theorem]
\label{thm:r/e}
For $I\subset \tilde{I}$,
assume that $B$ and $B'$ are
skew symmetric matrices such that
$B$ is the $I$-restriction of $\tilde{B}$
and  $\tilde{B}$ is an $\tilde{I}$-extension of $B$.
\par
(a) (Restriction)
Suppose that
an $I$-sequence
 $\mathbf{i}=(i_1,\dots,i_r)$ is a period 
of $(\tilde{B},\tilde{x},\tilde{y})$ in
 $\mathcal{A}(\tilde{B},\tilde{x},\tilde{y})$,
Then, $\mathbf{i}$ is also a period 
of $(B,x,y)$  in
 $\mathcal{A}(B,x,y)$.
\par
(b) (Extension)
Suppose that
an $I$-sequence
 $\mathbf{i}=(i_1,\dots,i_r)$ is a period 
of $(B,x,y)$  in
 $\mathcal{A}(B,x,y)$.
Then, $\mathbf{i}$ is also a period 
of $(\tilde{B},\tilde{x},\tilde{y})$ in
 $\mathcal{A}(\tilde{B},\tilde{x},\tilde{y})$.
\end{Theorem}

\begin{proof}
\par
(a) This is a trivial part.
In general,
the seed
$(B',x',y')=\mu_{\mathbf{i}}(B,x,y)$
is obtained from
the seed
$
(\tilde{B}',\tilde{x}',\tilde{y}')
=
\mu_{\mathbf{i}}(\tilde{B},\tilde{x},\tilde{y})$
by
restricting the index set $\tilde{I}$
of $(\tilde{B}',\tilde{x}',\tilde{y}')$ to $I$
{\em and\/} specializing the `external' initial cluster variables
$x_j$ ($j\in \tilde{I}\setminus I$) to $1$ appearing
in the `internal' cluster variables $\tilde{x}'_i$ ($i\in I$).
Thus, the periodicity of $(B,x,y)$
follows from the periodicity of $(\tilde{B},\tilde{x},\tilde{y})$.
\par
(b) Set  $(\tilde{B}',\tilde{x}',
\tilde{y}'):=\mu_{\mathbf{i}}
(\tilde{B},\tilde{x},\tilde{y})$.
Thanks to Theorem \ref{thm:tropperiod},
it is enough to prove 
\begin{align}
\label{eq:ytp}
[\tilde{y}'_i]_{\mathbf{T}}=
\tilde{y}_i
\quad (i\in \tilde{I}).
\end{align}
Since the `external' coefficients $\tilde{y}'_j$
($j\in \tilde{I}\setminus I$)
do not influence each other
under the mutation $\mu_{\mathbf{i}}$,
it is enough to verify \eqref{eq:ytp}
when there is {\em only one\/} external
  index $j\in \tilde{I}\setminus I$.
Thus, we may assume $I=\{1,\dots,n\}$ and $\tilde{I}= \{1,
\dots,n+1\}$.
\par
By the periodicity assumption
 and the exchange relations
\eqref{eq:coef} and \eqref{eq:clust}, we know
{\em a priori\/}
 the following periodicities hold.
\begin{align}
\label{eq:yp}
[\tilde{y}'_i]_{\mathbf{T}}&=
\tilde{y}_i\quad
(i\in I),\\
\label{eq:xp}
\tilde{x}'_{n+1}&=\tilde{x}_{n+1}.
\end{align}
Let us show that \eqref{eq:yp} and \eqref{eq:xp} imply
\eqref{eq:ytp}.
We introduce the integers
$c'_{ij}$ ($i,j\in \tilde{I}$) by
\begin{align}
\label{eq:cint}
[\tilde{y}'_i]_{\mathbf{T}}&=
\prod_{j=1}^{n+1} 
\tilde{y}_j^{c_{ji}'}.
\end{align}
Let $\mathbf{g}'_i=(g'_{ji})_{j=1}^{n+1}$ be the $g$-vector
for $\tilde{x}'_i$ ($i\in \tilde{I}$) in \eqref{eq:gF}.
Then, by Theorem \ref{thm:CG}
the transpose of $C'=(c'_{ij})_{i,j=1}^{n+1}$ 
and $G'=(g'_{ij})_{i,j=1}^{n+1}$ are inverse to each other,
i.e.,
\begin{align}
\label{eq:inv}
{C'}^T G'=I.
\end{align}
Meanwhile, the equalities
\eqref{eq:yp} and \eqref{eq:xp} imply that ${C'}^T$ and $G'$ have the
following form:
\begin{align}
{C'}^T=
\left(
\begin{array}{ccc|c}
1&&&0\\
&\ddots&&\vdots\\
&&1&0\\
\hline
*&\cdots&*&*\\
\end{array}
\right),
\quad
G'=
\left(
\begin{array}{ccc|c}
*&\cdots&*&0\\
\vdots&&\vdots&\vdots\\
*&\cdots&*&0\\
\hline
*&\cdots&*&1\\
\end{array}
\right).
\end{align}
The condition \eqref{eq:inv}
further imposes that the matrix has the form
\begin{align}
{C'}^T=
\left(
\begin{array}{ccc|c}
1&&&0\\
&\ddots&&\vdots\\
&&1&0\\
\hline
\alpha_1&\cdots&\alpha_n&1\\
\end{array}
\right),
\quad
G'=
\left(
\begin{array}{ccc|c}
1&&&0\\
&\ddots&&\vdots\\
&&1&0\\
\hline
\beta_1&\cdots&\beta_{n}&1\\
\end{array}
\right).
\end{align}
with $\alpha_i=-\beta_i$.
Now we recall Theorem \ref{thm:pos}.

(a) The Laurent monomial \eqref{eq:cint}
is either positive or negative.
So we have $\alpha_i \geq 0$.

(b) For each $i$, $g'_{ij}$ ($j=1,\dots,n+1$)
are simultaneously nonpositive or nonnegative.
So we have $\beta_i \geq 0$.

Thus,  we have $\alpha_i=\beta_i=0$
and we conclude that $C'=G'=I$;
therefore, \eqref{eq:ytp} holds.
\end{proof}

\begin{Remark}
The extension theorem (b)
was partially formulated by Keller \cite{Keller08b},
where only the periodicity of $\tilde{B}$ was considered
for the periodicities of seeds for `pairs of Dynkin diagrams'
 studied in \cite{Keller10}.
His result motivates us to formulate
Theorem \ref{thm:r/e} (b).
A generalized version of the result of \cite{Keller08b}
will be contained in \cite{Plamondon10c}.
\end{Remark}

Theorem \ref{thm:r/e} tells that
one may assume
without losing much generality
that the components of a period $\mathbf{i}$
of a seed under study  exhaust
the index set $I$.

It may be  worth to mention that
the involution property of a mutation $\mu_i$
is also regarded as the simplest case of
Extension Theorem applied for the 
subquiver of type $A_1$ consisting of a single vertex $i$.

\section{T-systems and Y-systems}
\label{sec:T}
In this section,
we define the T- and Y-systems
associated with any  period of
an exchange matrix.
We do it in two steps.
First, we treat the special case when
the period is `regular'.
The corresponding T- and Y-systems 
 are natural generalizations of the known `classic' T-
and Y-systems.
Next, the notions of T- and Y-systems are
further extended to  the general case.
The latter ones will be used in Section \ref{sec:dilog}.
We stress that in this section  we do {\em not\/}  assume the
periodicity of seeds.

\subsection{Regular period}

Let $B$ a skew symmetrizable matrix with index set $I$.

\begin{Definition}
\label{def:regular}
We say that a $\nu$-period
$\mathbf{i}$ of $B$
 is {\em regular\/} if
it is a complete system of representatives
of the $\nu$-orbits in $I$;
in other words,  it satisfies
the following conditions:
\begin{itemize}
\item[(A1)] All the components of $\mathbf{j}(\mathbf{i},\nu)=\mathbf{i}
\,|\, \nu(\mathbf{i})\,|\, \cdots \,|\, \nu^{g-1}(\mathbf{i})$
exhaust $I$,
where $g$ is the order of $\nu$.
\item[(A2)] All the components of $\mathbf{i}$ belong
to  distinct $\nu$-orbits in $I$.
\end{itemize}
\end{Definition}

\begin{Example}
\label{ex:adm}
 In the previous examples,
we have the following regular $\nu$-periods of $B$ or $Q$.

In Example \ref{ex:preperiod},
$\mathbf{i}_+\, | \, \mathbf{i}_-$ is a regular period of $B$.

In Example \ref{ex:G5},
$\mathbf{i}^{\bullet}_+\,|\,
 \mathbf{i}^{\circ}_{+,1} \,|\,
\mathbf{i}^{\bullet}_-\,|\,
 \mathbf{i}^{\circ}_{+,4}
$
is a regular $\nu$-period of $Q$.

In Example \ref{ex:fordy},
$(1,2)$ 
is a regular $\rho^2$-period of $Q$.

\par
In Example \ref{ex:finite} (a),
$\mathbf{i}_+ \,  | \, \mathbf{i}_-$ is a regular period of $Q$.
\par
In Example \ref{ex:finite} (b),
$\mathbf{i}_+$ is a regular $\nu$-period of $Q$.
\par
In Example \ref{ex:affine} (a),
$\mathbf{i}_+$ is a regular $\nu$-period of $Q$.
\par
In Example \ref{ex:affine} (b),
$\mathbf{i}^{\bullet}_+
\, | \, \mathbf{i}^{\circ}_+\, | \,
\mathbf{i}^{\bullet}_-$ is a regular $\nu$-period of $Q$.
\par
In Example \ref{ex:sine},
$\mathbf{i}^{\bullet}_+
\, | \, (1 )\, | \,
\mathbf{i}^{\bullet}_-$ is a regular $\nu$-period of $Q$.
\end{Example}

%\begin{Remark}
%When the periodicity of the cluster algebra $\mathcal{A}(B,x,y)$
%is an issue,
%we do not lose generality by the condition (A1).
%Indeed,  one can reduce the index set $I$
%by the restriction in Theorem \ref{thm:r/e} to fulfill (A1).
%On the other hand, by condition (A2) we {\em may\/} lose
%generality, even though all the examples of periodic cluster
%algebras we have known so far admit suitable
%regular $\nu$-periods of $B$ as in the above examples.
%We leave the investigation of this issue as a future problem.
%\end{Remark}

Suppose that $\mathbf{i}=(i_1,\dots,i_r)$ is a $\nu$-period
of $B$.
Let us decompose $\mathbf{i}$ into $t$ parts as follows:
\begin{align}
\label{eq:slice}
\mathbf{i}&=\mathbf{i}(0)\,|\,  \mathbf{i}(1)\,|\, 
\cdots \,|\, \mathbf{i}(t-1),\\
\mathbf{i}(p)&=
(i(p)_1,\dots,i(p)_{r_p}),
\quad
\sum_{p=0}^{t-1} r_{p} = r.
\end{align}
Let $B(p)$ ($p=0,\dots,t-1$) be the matrices
defined by the sequence of 
mutations
\begin{align}
\label{eq:Bmutseq}
B(0)=B
\
\mathop{\longrightarrow}^{\mu_{\mathbf{i}(0)}}
\
B(1)
\
\mathop{\longrightarrow}^{\mu_{\mathbf{i}(1)}}
\
\cdots
\
\mathop{\longrightarrow}^{\mu_{\mathbf{i}(t-1)}}
\
B(t)=\nu(B),
\end{align}
where $\nu(B)=(b'_{ij})$ is the matrix defined by
$b'_{\nu(i)\nu(j)}=b_{ij}$.

\begin{Definition}
A decomposition \eqref{eq:slice}
  of a  $\nu$-period 
$\mathbf{i}$ of $B$ is called a  {\em slice\/} of $\mathbf{i}$
(of length $t$) if it satisfies
the following condition:
\begin{align}
\label{eq:compat}
b(p)_{i(p)_a, i(p)_b}=0
\quad (1\leq a,b \leq r_p)\quad
\mbox{for any $p=0,\dots,t-1$.}
\end{align}
\end{Definition}

For any $\mathbf{i}$,
there is at least one slice of $\mathbf{i}$,
i.e., the one
 $\mathbf{i}=( i_1)\,|\,
(i_2) \,|\, \cdots \,|\, ( i_r)$ of maximal length.
In general, there may be several slices of $\mathbf{i}$.

If
$\mathbf{i}=\mathbf{i}(0)\,|\, 
\cdots \,|\, \mathbf{i}(t-1)$
is a slice of a $\nu$-period of $B$,
 then, due to Lemma \ref{lem:order},
each composite mutation $\mu_{\mathbf{i}(p)}$
in \eqref{eq:Bmutseq} does not depend on the order
of the sequence $\mathbf{i}(p)$ and it is involutive.
Furthermore, the sequence  \eqref{eq:Bmutseq}
is extended to the infinite one
\begin{align}
\label{eq:Bmutseq1}
\begin{split}
&\hskip100pt
\cdots
\
\mathop{\longleftrightarrow}^{\mu_{\nu^{-1}(\mathbf{i}(t-2))}}
\
B(-1)
\
\mathop{\longleftrightarrow}^{\mu_{\nu^{-1}(\mathbf{i}(t-1))}}
\\
&B(0)
\
\mathop{\longleftrightarrow}^{\mu_{\mathbf{i}(0)}}
\
B(1)
\
\mathop{\longleftrightarrow}^{\mu_{\mathbf{i}(1)}}
\
\cdots
\
\mathop{\longleftrightarrow}^{\mu_{\mathbf{i}(t-2)}}
\
B(t-1)
\mathop{\longleftrightarrow}^{\mu_{\mathbf{i}(t-1)}}
\\
&B(t)
\
\mathop{\longleftrightarrow}^{\mu_{\nu(\mathbf{i}(0))}}
B(t+1)
\
\mathop{\longleftrightarrow}^{\mu_{\nu(\mathbf{i}(1))}}
\
\cdots
\
\mathop{\longleftrightarrow}^{\mu_{\nu(\mathbf{i}(t-2))}}
\
B(2t-1)
\mathop{\longleftrightarrow}^{\mu_{\nu(\mathbf{i}(t-1))}}
\\
&B(2t)
\
\mathop{\longleftrightarrow}^{\mu_{\nu^2(\mathbf{i}(0))}}
\
\cdots,
\end{split}
\end{align}
where $B(nt)=\nu^n(B)$ for $n\in \mathbb{Z}$.
In particular,
$B(gt)=B$, where $g$ is the order of $\nu$.
Thus, the sequence \eqref{eq:Bmutseq1}
has a period $gt$ with respect to $u$.

\begin{Example}
\label{ex:slice}
For the regular $\nu$-periods of $B$ or $Q$ in Examples \ref{ex:adm},
we have the following slices.

In Example \ref{ex:preperiod},
$\mathbf{i}_+\, |\,
\mathbf{i}_-$ itself is a slice  of length 2,
and $gt=2$.
\par
In Example \ref{ex:G5},
$
\mathbf{i}^{\bullet}_+\,|\,
 \mathbf{i}^{\circ}_{+,1} \,|\,
\mathbf{i}^{\bullet}_-\,|\,
 \mathbf{i}^{\circ}_{+,4}
=
(\mathbf{i}^{\bullet}_+\,|\,
 \mathbf{i}^{\circ}_{+,1}) \,|\,
(\mathbf{i}^{\bullet}_-\,|\,
 \mathbf{i}^{\circ}_{+,4})
$
is  a slice  of length 2,
and $gt=10$.
\par
In Example \ref{ex:fordy},
$(1,2)=(1)\, | \, (2) $ is a slice of length 2, and $gt=6$.
\par
In Example \ref{ex:finite} (a),
$\mathbf{i}_+\, |\,
\mathbf{i}_-$ itself is  a slice  of length 2,
and $gt=2$.
\par
In Example \ref{ex:finite} (b),
$\mathbf{i}_+$ itself is  a slice  of length 1,
and $gt=2$.
\par
In Example \ref{ex:affine} (a),
$\mathbf{i}_+$ itself is a slice of length 1,
and $gt=2$.
\par
In Example \ref{ex:affine} (b),
$\mathbf{i}^{\bullet}_+
\, | \, \mathbf{i}^{\circ}_+\, | \,
\mathbf{i}^{\bullet}_-=
(\mathbf{i}^{\bullet}_+
\, | \, \mathbf{i}^{\circ}_+)\, | \,
\mathbf{i}^{\bullet}_-$ is
a slice  of length 2,
and $gt=4$.
\par
In Example \ref{ex:sine},
$
\mathbf{i}^{\bullet}_+
\, | \, (1 )\, | \,
\mathbf{i}^{\bullet}_-
=
(\mathbf{i}^{\bullet}_+
\, | \, (1 ))\, | \,
\mathbf{i}^{\bullet}_-$ is
a slice  of length 2,
and $gt=12$.
\end{Example}

\subsection{T- and Y-systems for regular period}
\label{subsec:TY1}

Here we introduce the T- and Y-systems for {\em regular\/} $\nu$-periods,
which are especially important in applications.
The T- and Y-systems for general $\nu$-periods
 will be treated in Section \ref{subsec:nonadm}.

Assume that
$\mathbf{i}=\mathbf{i}(0)\,|\, 
\cdots \,|\, \mathbf{i}(t-1)$
is a slice of a regular $\nu$-period of $B$.

In view of \eqref{eq:Bmutseq1},
we set $(B(0),x(0),y(0)):=(B,x,y)$ (the initial seed
of $\mathcal{A}(B,x,y)$),
and  consider the corresponding infinite
sequence of mutations of {\em seeds\/}
\begin{align}
\label{eq:seedmutseq}
\begin{split}
\cdots
\mathop{\longleftrightarrow}^{\mu_{\nu^{-1}(\mathbf{i}(t-2))}}
\
&(B(-1),x(-1),y(-1))
\
\mathop{\longleftrightarrow}^{\mu_{\nu^{-1}(\mathbf{i}(t-1))}}
\
(B(0),x(0),y(0))
\
\mathop{\longleftrightarrow}^{\mu_{\mathbf{i}(0)}}
\\
&(B(1),x(1),y(1))
\
\mathop{\longleftrightarrow}^{\mu_{\mathbf{i}(1)}}
\
(B(2),x(2),y(2))
\
\mathop{\longleftrightarrow}^{\mu_{\mathbf{i}(2)}}
\
\cdots,
\end{split}
\end{align}
thereby introducing a family of clusters $x(u)$ ($u\in \mathbb{Z}$)
and coefficients tuples $y(u)$ ($u\in \mathbb{Z}$).

We define a subset $P_+$ of $I\times \mathbb{Z}$ by
$(i,u)\in P_+$
if  and only if $i$ is a component of $\nu^{m}(\mathbf{i}(k))$
for $u=mt+k$ ($m\in \mathbb{Z}, 0\leq k\leq t-1$).
Plainly speaking, $(i,u)\in P_+$ 
 is a {\em forward mutation point\/}
in \eqref{eq:seedmutseq}.
Similarly, we 
define a subset $P_-$ of $I\times \mathbb{Z}$ by
$(i,u)\in P_-$
if and only if $(i,u-1)\in P_+$, namely,
$(i,u)\in P_-$ is a {\em backward mutation point\/}
in \eqref{eq:seedmutseq}. Below we mainly use $P_+$.
(Alternatively, one may use $P_-$ throughout.)

Let $g$ be the order of $\nu$.
For each $i\in I$, let $g_i$ be the smallest positive
integer such that $\nu^{g_i}(i)=i$.
Therefore, $g_i$ is a divisor of $g$.
Note that $(i,u)\in P_+$
if and only if $(i,u+tg_i)\in P_+$.
It is convenient to define a subset
$\tilde{P}_+$ of $I \times \frac{1}{2}\mathbb{Z}$
by  $(i,u)\in \tilde{P}_+$
if and only if $(i,u+\frac{tg_i}{2})\in P_+$.
Consequently, we have
\begin{align}
\textstyle
(i,u) \in \tilde{P}_+
\Longleftrightarrow 
(i,u\pm\frac{tg_i}{2}) \in P_+.
\end{align}

First, we explain what is the Y-system, in short.
The sequence of mutations \eqref{eq:seedmutseq} gives
various relations among coefficients $y_i(u)$ ($(i,u)\in I 
\times \mathbb{Z}$) by the exchange relation
\eqref{eq:coef}.
Then,
thanks to the assumption (A1) in Definition 
\ref{def:regular},
all these coefficients are 
 products of the
`generating' coefficients $y_i(u)$ and $1+y_i(u)$ ($(i,u)\in P_+$)
and their inverses.
Here we also used the fact
that $y_i(u)=y_i(u-1)^{-1}$ for $(i,u)\in P_-$.
Furthermore, these generating coefficients
obey some relations,
which are the Y-system.

\begin{figure}
\begin{center}
\begin{picture}(340,75)(0,5)
\put(0,0){
% horizontal
\drawline(0,20)(150,20)
\drawline(0,80)(150,80)
\drawline(30,50)(120,50)
% vertical
\drawline(30,20)(30,22)
\drawline(30,29)(30,31)
\drawline(30,39)(30,41)
\drawline(30,49)(30,51)
\drawline(30,59)(30,61)
\drawline(30,69)(30,71)
\drawline(30,79)(30,80)
\drawline(120,20)(120,22)
\drawline(120,29)(120,31)
\drawline(120,39)(120,41)
\drawline(120,49)(120,51)
\drawline(120,59)(120,61)
\drawline(120,69)(120,71)
\drawline(120,79)(120,80)
\drawline(60,50)(60,65)
\drawline(90,50)(90,35)
%
% cross point
\drawline(28,48)(32,52)
\drawline(32,48)(28,52)
\drawline(118,48)(122,52)
\drawline(122,48)(118,52)
\drawline(58,63)(62,67)
\drawline(62,63)(58,67)
\drawline(88,33)(92,37)
\drawline(92,33)(88,37)
% dot point
\put(60,50){\circle*{3}}
\put(90,50){\circle*{3}}
\put(28,5){\small $u$}
\put(108,5){\small $u+tg_i$}
\put(18,47){\small $i$}
\put(48,62){\small $j$}
}
\put(190,0){
% horizontal
\drawline(0,20)(150,20)
\drawline(0,80)(150,80)
\drawline(30,50)(120,50)
\drawline(10,65)(60,65)
\drawline(20,35)(90,35)
% vertical
\drawline(30,20)(30,22)
\drawline(30,29)(30,31)
\drawline(30,39)(30,41)
\drawline(30,49)(30,51)
\drawline(30,59)(30,61)
\drawline(30,69)(30,71)
\drawline(30,79)(30,80)
\drawline(120,20)(120,22)
\drawline(120,29)(120,31)
\drawline(120,39)(120,41)
\drawline(120,49)(120,51)
\drawline(120,59)(120,61)
\drawline(120,69)(120,71)
\drawline(120,79)(120,80)
\drawline(30,50)(30,65)
\drawline(30,50)(30,35)
%
% cross point
\drawline(28,48)(32,52)
\drawline(32,48)(28,52)
\drawline(118,48)(122,52)
\drawline(122,48)(118,52)
\drawline(58,63)(62,67)
\drawline(62,63)(58,67)
\drawline(8,63)(12,67)
\drawline(12,63)(8,67)
\drawline(88,33)(92,37)
\drawline(92,33)(88,37)
\drawline(18,33)(22,37)
\drawline(22,33)(18,37)
% dot point
\put(30,65){\circle*{3}}
\put(30,35){\circle*{3}}
\put(28,5){\small $u$}
\put(108,5){\small $u+tg_i$}
\put(18,47){\small $i$}
\put(-2,62){\small $j$}
}
\end{picture}
\end{center}
\caption{Schematic diagrams for Y-system (left) and T-system (right).
A marked point represents a forward mutation point.
A vertical bond represents the fact $b_{ki}(v)\neq 0$.
}
\label{fig:diagram}
\end{figure}
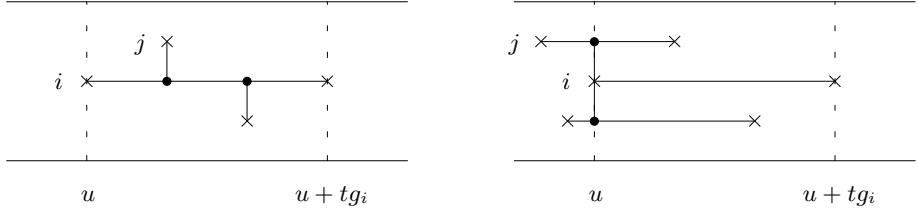
Let us write down the  relations explicitly.
Take $(i,u)\in P_+$ and consider the mutation
at $(i,u)$,
where $y_i(u)$ is exchanged to
$y_i(u+1)=y_i(u)^{-1}$.
Then, for each $(j,v)\in P_+$ such that
 $v\in (u,u+tg_i)$
(i.e., $u <v < u+tg_i$),
$y_i(v)$ are multiplied by factors $(1+y_j(v))^{-b_{ji}(v)}$ for
 $b_{ji}(v)<0$ and
$(1+y_j(v)^{-1})^{-b_{ji}(v)}$ for $b_{ji}(v)>0$.
The result coincides with the coefficient $y_i(u+tg_i)$
by the assumption (A2) in Definition \ref{def:regular}.
See Figure \ref{fig:diagram}.
In summary, we have the following relations:
For  $(i,u)\in P_+$,
\begin{align}
\label{eq:yi'}
\textstyle
y_i\left(u \right)y_i\left(u+tg_i \right)
&=
\frac{
\displaystyle
\prod_{(j,v)\in P_+} (1+y_j(v))^{G'_+(j,v;i,u)}
}
{
\displaystyle
\prod_{(j,v)\in P_+} (1+y_j(v)^{-1})^{G'_-(j,v;i,u)}
},\\
\label{eq:G'}
G'_{\pm}(j,v;i,u)&=
\begin{cases}
\mp b_{ji}(v) & v\in (u, u+tg_i),
b_{ji}(v)\lessgtr 0\\
0 & \mbox{otherwise}.
\end{cases}
\end{align}
Or, equivalently, for  $(i,u)\in \tilde{P}_+$,
\begin{align}
\label{eq:yi}
\textstyle
y_i\left(u-\frac{tg_i}{2}\right)y_i\left(u+\frac{tg_i}{2}\right)
&=
\frac{
\displaystyle
\prod_{(j,v)\in P_+} (1+y_j(v))^{G_+(j,v;i,u)}
}
{
\displaystyle
\prod_{(j,v)\in P_+} (1+y_j(v)^{-1})^{G_-(j,v;i,u)}
},\\
\label{eq:G}
G_{\pm}(j,v;i,u)&=
\begin{cases}
\mp b_{ji}(v) & v\in (u-\frac{tg_i}{2}, u+\frac{tg_i}{2}),
b_{ji}(v)\lessgtr 0\\
0 & \mbox{otherwise}.
\end{cases}
\end{align}
We call the system of relations \eqref{eq:yi}
 the
{\em Y-system associated with a slice 
$\mathbf{i}=\mathbf{i}(0)\,|\, 
\cdots \,|\, \mathbf{i}(t-1)$
 of a regular $\nu$-period of $B$}.

Next, we explain what is the T-system, in short.
The sequence of mutations \eqref{eq:seedmutseq} gives
various relations among cluster variables $x_i(u)$ ($(i,u)\in I 
\times \mathbb{Z}$) by the exchange relation
\eqref{eq:clust}.
Again, thanks to the assumption (A1) in Definition 
\ref{def:regular},
all these coefficients are represented 
by the
`generating' cluster variables $x_i(u)$ ($(i,u)\in P_+$).
Furthermore, these generating cluster variables
obey some relations,
which are the T-system.

Let us write down the relations explicitly.
Take $(i,u)\in P_+$ and consider the mutation
at $(i,u)$.
Then, by \eqref{eq:clust}
and the assumption (A2) in Definition \ref{def:regular},
 we have
\begin{align}
\label{eq:xi'}
\begin{split}
x_i(u)x_i(u+tg_i)
&=
\frac{y_i(u)}{1+y_i(u)}
\prod_{(j,v)\in P_+} x_j(v)^{H'_+(j,v;i,u)}\\
&\quad
+
\frac{1}{1+y_i(u)}
\prod_{(j,v)\in P_+} x_j(v)^{H'_-(j,v;i,u)},
\end{split}\\
H'_{\pm}(j,v;i,u)&=
\begin{cases}
\pm b_{ji}(u)&
u\in (v- tg_j,v), b_{ji}(u)\gtrless 0\\
0 & \mbox{otherwise}.
\end{cases}
\end{align}
See Figure \ref{fig:diagram}.
By introducing the `shifted cluster variables'
$\tilde{x}_i(u):=x_i(u+\frac{tg_i}{2})$ for $(i,u)\in \tilde{P}_+$,
these relations can be written 
 in a more `balanced' form
and become parallel to \eqref{eq:yi}
as follows:
For $(i,u)\in P_+$,
\begin{align}
\label{eq:xi}
\begin{split}
\textstyle
\tilde{x}_i(u-\frac{tg_i}{2})\tilde{x}_i(u+\frac{tg_i}{2})
&=
\frac{y_i(u)}{1+y_i(u)}
\prod_{(j,v)\in\tilde{P}_+} \tilde{x}_j(v)^{H_+(j,v;i,u)}\\
&\quad
+
\frac{1}{1+y_i(u)}
\prod_{(j,v)\in \tilde{P}_+} \tilde{x}_j(v)^{H_-(j,v;i,u)},
\end{split}\\
\label{eq:tH}
H_{\pm}(j,v;i,u)&=
\begin{cases}
\pm b_{ji}(u)&
u\in (v- \frac{tg_j}{2},v+\frac{tg_j}{2}),
 b_{ji}(u)\gtrless 0\\
0 & \mbox{otherwise}.
\end{cases}
\end{align}
We call the system of relations \eqref{eq:xi}
 the
{\em T-system (with coefficients)
associated with a slice 
$\mathbf{i}=\mathbf{i}(0)\,|\, 
\cdots \,|\, \mathbf{i}(t-1)$
 of a regular $\nu$-period of $B$}.

Let $\mathcal{A}(B,x)$ be the cluster algebra with
trivial coefficients with initial seed $(B,x)$.
Namely, we set every coefficient to be $1$
in the trivial semifield $\mathbf{1}=\{1\}$.
Let $\pi_{\mathbf{1}}:\mathbb{P}_{\mathrm{univ}}(y)
\rightarrow \mathbf{1}$ be the projection.
Let $[x_i(u)]_{\mathbf{1}}$ be the image
of $x_i(u)$ by the algebra homomorphism
$\mathcal{A}(B,x,y)\rightarrow \mathcal{A}(B,x)$
induced from $\pi_{\mathbf{1}}$.
By the specialization of \eqref{eq:xi},
we have
\begin{align}
\label{eq:ti}
\textstyle
[\tilde{x}_i(u-\frac{tg_i}{2})]_{\mathbf{1}}
[\tilde{x}_i(u+\frac{tg_i}{2})]_{\mathbf{1}}
&=
\prod_{(j,v)\in\tilde{P}_+}
 [\tilde{x}_j(v)]_{\mathbf{1}}^{H_+(j,v;i,u)}
+
\prod_{(j,v)\in \tilde{P}_+}
 [\tilde{x}_j(v)]_{\mathbf{1}}^{H_-(j,v;i,u)}.
\end{align}
We call the system of relations \eqref{eq:ti}
 the
{\em T-system (without coefficients)
associated with a slice
$\mathbf{i}=\mathbf{i}(0)\,|\, 
\cdots \,|\, \allowbreak\mathbf{i}(t-1)$
 of a regular $\nu$-period of $B$}.

Needless to say, the unbalanced
form \eqref{eq:yi'} and \eqref{eq:xi'}
may be also useful in some situation.

\begin{Remark}
\label{rem:choice}
For a given regular $\nu$-period $\mathbf{i}$,
different choices of slices
of $\mathbf{i}$ give different Y-systems/T-systems.
But they are easily identified by `change of variables',
so the choice of a slice is not essential.
(See Proposition \ref{prop:choice} for a more precise statement.)
However, in view of the sequence \eqref{eq:seedmutseq},
it is economical, and often natural, to use
a slice {\em whose length $t$ is minimal\/} among all the other slices,
as in Example \ref{ex:slice}.
\end{Remark}

One may think that
the Y- and T-systems are
the both sides of the coin in the sense that
they directly determine each other as follows.

\begin{Proposition}[Duality]
\label{prop:TY1}
Let $D=\mathrm{diag} (d_i)_{i\in I}$ be
a diagonal matrix such that ${}^{t}(DB)=-DB$.
Let $G_{\pm}(j,v;i,u)$ and 
$H_{\pm}(j,v;i,u)$ the ones
in \eqref{eq:G}
and \eqref{eq:tH}, respectively.
Then, the following relation holds
for any $(j,v)\in P_+$ and $(i,u)\in \tilde{P}_+$.
\begin{align}
\label{eq:GH1}
d_j G_{\pm} (j,v;i,u)= d_i H_{\pm} (i,u;j,v).
\end{align}
In particular, if $B$ is skew symmetric, we have
\begin{align}
\label{eq:GH2}
 G_{\pm} (j,v;i,u)= H_{\pm} (i,u;j,v).
\end{align}
\end{Proposition}
\begin{proof}
We recall that any matrix $B'$ obtained from $B$
by mutation shares the same diagonalizing matrix $D$
with $B$ \cite[Proposition 4.5]{Fomin02}.
Then, by comparing \eqref{eq:G} and \eqref{eq:tH},
we immediately obtain the claim.
\end{proof}

\subsection{Examples}
\label{subsec:exTY}

Let us write down the Y- and T-systems explicitly
for the ones in Example \ref{ex:affine}.

(a) $(X,\ell)=(A_4,4)$.
 Let $Q$, $\mathbf{i}_+$, and $\nu$ be the one
therein.
Then, $\mathbf{i}_+$ is a regular $\nu$-period of $Q$.
We regard $\mathbf{i}_+$ as a slice of itself of length 1,
and consider the associated Y- and T-systems.
We use the index set $I=\{1,2,3,4\}\times \{1,2,3\}$
such that $(i,j)\in I$ corresponds to the vertex
at the $i$th column (from the left) and the $j$th row
(from the bottom).
Thus, we have the data $t=1$, $g=2$,
and $g_{(i,j)} =2$ for any $(i,j)\in I$.
The condition for the forward mutation points is
given by
\begin{align}
((i,j),u)\in P_+
\ &\Longleftrightarrow \
\mbox{$i+j+u$ is even}.
\end{align}
By writing $y_{(i,j)}(u)$ and
$\tilde{x}_{(i,j)}(u)$ as $y_{i,j}(u)$ and $\tilde{x}_{i,j}(u)$, 
the resulting Y-system and T-system (without coefficients) are as follows.

{\em Y-system:} For $((i,j),u)\in \tilde{P}_+$,
\begin{align}
\label{eq:YA}
y_{i,j}(u-1)y_{i,j}(u+1)
&= \frac{
(1+y_{i-1,j}(u))(1+y_{i+1,j}(u))
}
{
(1+y_{i,j-1}(u)^{-1})(1+y_{i,j+1}(u)^{-1})
},
\end{align}
where $y_{0,j}(u)=y_{5,j}(u)=0$
and $y_{i,0}(u)^{-1}=y_{i,4}(u)^{-1}=0$ in the right hand side.

{\em T-system:} For $((i,j),u)\in P_+$,
\begin{align}
\label{eq:TA}
[\tilde{x}_{i,j}(u-1)]_{\mathbf{1}}[\tilde{x}_{i,j}(u+1)]_{\mathbf{1}}
&= 
[\tilde{x}_{i-1,j}(u)]_{\mathbf{1}}[\tilde{x}_{i+1,j}(u)]_{\mathbf{1}}
+
[\tilde{x}_{i,j-1}(u)]_{\mathbf{1}}[\tilde{x}_{i,j+1}(u)]_{\mathbf{1}}
,
\end{align}
where $[\tilde{x}_{0,j}(u)]_{\mathbf{1}}
=[\tilde{x}_{5,j}(u)]_{\mathbf{1}}=
[\tilde{x}_{i,0}(u)]_{\mathbf{1}}
=[\tilde{x}_{i,4}(u)]_{\mathbf{1}}=1$ in the right hand side.

These are the Y- and T-systems associated
with the quantum affine algebras of type $A_4$ with level 4.
The relation \eqref{eq:TA} is also a special case of Hirota's
bilinear difference equation \cite{Hirota77},
and it is one of the most studied difference equations in various
view points.
See \cite{Inoue10c} for more information.

(b) $(X,\ell)=(B_4,4)$.
 Let $Q$, $\mathbf{i}^{\bullet}_+$, $\mathbf{i}^{\bullet}_-$,
$\mathbf{i}^{\circ}_+$,
$\mathbf{i}^{\circ}_-$, and $\nu$ be the one
therein.
Then, $(\mathbf{i}^{\bullet}_+\, | \,
\mathbf{i}^{\circ}_+\,) | \,
\mathbf{i}^{\bullet}_-$ is a regular $\nu$-period of $Q$.
We regard it as a slice of itself of length 2,
and consider the associated Y- and T-systems.
We use the index set $I$ which is the disjoint
union of $\{1,2,3,5,6,7\}\times \{1,2,3\}$
and $\{4\} \times \{1,\dots,7\}$ such that
$(i,j)\in I$ corresponds to the vertex
at the $i$th column (from the left) and the $j$th row
(from the bottom).
Thus, we have the data $t=2$, $g=2$,
and $g_{(i,j)} =2$ for $i\neq 4$ and $g_{(4,j)} =1$.
The condition for the forward mutation points is
given by
\begin{align}
((i,j),u)\in P_+
\ &\Longleftrightarrow \
\begin{cases}
(i,j)\in \mathbf{i}^{\bullet}_+ 
\sqcup \mathbf{i}^{\circ}_+
& u \equiv 0 \ (4)\\
(i,j)\in \mathbf{i}^{\bullet}_- 
& u \equiv 1,3 \ (4)\\
(i,j)\in \mathbf{i}^{\bullet}_+ 
\sqcup \mathbf{i}^{\circ}_-
& u \equiv 2 \ (4).\\
\end{cases}
\end{align}
The resulting Y-system and T-system (without coefficients) are as follows,
where the `boundary terms'
in the right hand should be ignored as before.

{\em Y-system:}  For $((i,j),u)\in \tilde{P}_+$ with
$i=1,2,6,7$, 
\begin{align}
y_{i,j}(u-2)y_{i,j}(u+2)
&= \frac{
(1+y_{i-1,j}(u))(1+y_{i+1,j}(u))
}
{
(1+y_{i,j-1}(u)^{-1})(1+y_{i,j+1}(u)^{-1})
},
\end{align}
and, with $i=3,4,5$,
\begin{align}
\begin{split}
y_{3,j}(u-2)y_{3,j}(u+2)
&= 
\frac
{
\genfrac{}{}{0pt}{}{
\displaystyle
(1+y_{2,j}(u))(1+y_{4,2j-1}(u))(1+y_{4,2j+1}(u))
}
{
\displaystyle
\times (1+y_{4,2j}(u-1))(1+y_{4,2j}(u+1))
}
}
{
(1+y_{3,j-1}(u)^{-1})(1+y_{3,j+1}(u)^{-1})
},\\
y_{4,2j}(u-1)y_{4,2j}(u+1)
&=
\begin{cases}
\dfrac
{
1+y_{3,j}(u)
}
{
(1+y_{4,2j-1}(u)^{-1})(1+y_{4,2j+1}(u)^{-1})
}& u+2j\equiv 0\ (4)
\\
\dfrac
{
1+y_{5,j}(u)
}
{
(1+y_{4,2j-1}(u)^{-1})(1+y_{4,2j+1}(u)^{-1})
}&u+2j\equiv 2\ (4),\\
\end{cases}
\\
y_{4,2j+1}(u-1)y_{4,2j+1}(u+1)
&= 
\frac
{
1
}
{
(1+y_{4,2j}(u)^{-1})(1+y_{4,2j+2}(u)^{-1})
},\\
y_{5,j}(u-2)y_{5,j}(u+2)
&= 
\frac
{
\genfrac{}{}{0pt}{}{
\displaystyle
(1+y_{6,j}(u))(1+y_{4,2j-1}(u))(1+y_{4,2j+1}(u))
}
{
\displaystyle
\times (1+y_{4,2j}(u-1))(1+y_{4,2j}(u+1))
}
}
{
(1+y_{5,j-1}(u)^{-1})(1+y_{5,j+1}(u)^{-1})
}.
\end{split}
\end{align}

{\em T-system:} For $((i,j),u)\in P_+$ with $i=1,2,6,7$,
\begin{align}
[\tilde{x}_{i,j}(u-2)]_{\mathbf{1}}[\tilde{x}_{i,j}(u+2)]_{\mathbf{1}}
&= 
[\tilde{x}_{i-1,j}(u)]_{\mathbf{1}}[\tilde{x}_{i+1,j}(u)]_{\mathbf{1}}
+
[\tilde{x}_{i,j-1}(u)]_{\mathbf{1}}[\tilde{x}_{i,j+1}(u)]_{\mathbf{1}}
,
\end{align}
and, with $i=3,4,5$,
\begin{align}
\begin{split}
&[\tilde{x}_{3,j}(u-2)]_{\mathbf{1}}[\tilde{x}_{3,j}(u+2)]_{\mathbf{1}}
= 
[\tilde{x}_{2,j}(u)]_{\mathbf{1}}[\tilde{x}_{4,2j}(u)]_{\mathbf{1}}
+
[\tilde{x}_{3,j-1}(u)]_{\mathbf{1}}[\tilde{x}_{3,j+1}(u)]_{\mathbf{1}}
,\\
&[\tilde{x}_{4,2j}(u-1)]_{\mathbf{1}}[\tilde{x}_{4,2j}(u+1)]_{\mathbf{1}}\\
&
= 
\begin{cases}
[\tilde{x}_{5,j}(u-1)]_{\mathbf{1}}[\tilde{x}_{3,j}(u+1)]_{\mathbf{1}}
+
[\tilde{x}_{4,2j-1}(u)]_{\mathbf{1}}[\tilde{x}_{4,2j+1}(u)]_{\mathbf{1}}
&u+2j\equiv 1\ (4)\\
[\tilde{x}_{3,j}(u-1)]_{\mathbf{1}}[\tilde{x}_{5,j}(u+1)]_{\mathbf{1}}
+
[\tilde{x}_{4,2j-1}(u)]_{\mathbf{1}}[\tilde{x}_{4,2j+1}(u)]_{\mathbf{1}}
&u+2j\equiv 3\ (4),\\
\end{cases}
\\
& [\tilde{x}_{4,2j+1}(u-1)]_{\mathbf{1}}[\tilde{x}_{4,2j+1}(u+1)]_{\mathbf{1}}
\\
&= 
 \begin{cases}
[\tilde{x}_{5,j}(u)]_{\mathbf{1}}[\tilde{x}_{3,j+1}(u)]_{\mathbf{1}}
+
[\tilde{x}_{4,2j}(u)]_{\mathbf{1}}[\tilde{x}_{4,2j+2}(u)]_{\mathbf{1}}
&u+2j\equiv 0\ (4)\\
[\tilde{x}_{3,j}(u)]_{\mathbf{1}}[\tilde{x}_{5,j+1}(u)]_{\mathbf{1}}
+
[\tilde{x}_{4,2j}(u)]_{\mathbf{1}}[\tilde{x}_{4,2j+2}(u)]_{\mathbf{1}}
&u+2j\equiv 2\ (4),\\
\end{cases}
\\
&[\tilde{x}_{5,j}(u-2)]_{\mathbf{1}}[\tilde{x}_{5,j}(u+2)]_{\mathbf{1}}
= 
[\tilde{x}_{6,j}(u)]_{\mathbf{1}}[\tilde{x}_{4,2j}(u)]_{\mathbf{1}}
+
[\tilde{x}_{5,j-1}(u)]_{\mathbf{1}}[\tilde{x}_{5,j+1}(u)]_{\mathbf{1}}
.
\end{split}
\end{align}

Under a suitable identification of variables,
they coincide with the Y- and T-systems associated
with the quantum affine algebras of type $B_4$ with level 4.
See \cite{Inoue10a} for more information.

\subsection{Standalone versions of T- and Y-systems}
\label{subsec:stand}

Here we establish certain formal property
concerning T- and Y-systems.

We just
introduced the T-systems (without coefficient, for simplicity)
\eqref{eq:ti} and Y-systems \eqref{eq:yi}
as relations inside a cluster algebra $\mathcal{A}(B,x)$
and a coefficient group $\mathcal{G}(B,y)$.
On the other hand, usually
these relations appear 
without their `ambient' cluster algebras or coefficient groups.
Therefore, to apply the cluster algebraic machinery
of \cite{Fomin07}, which has been proved to be so efficient and powerful,
it is necessary to establish 
a precise connection between such
`standalone' T- and Y-systems and
the `built-in' T- and Y-systems inside cluster algebras.
In fact, this has been repeatedly established
case-by-case for each explicit example
(e.g., \cite{Inoue10c,Kuniba09,Inoue10a,Nakanishi10a,Nakanishi10b}).
Here we do it again, and hopeful for the last time, in a general setting.

We continue to use the same notations as in Section \ref{subsec:TY1}.

First, we introduce the ring/group
associated with the built-in T-system/Y-system in
 cluster algebra/coefficient group.
\begin{Definition}
\label{def:BI}
(1)
 The {\em T-subalgebra $\mathcal{A}_{\mathbf{i}}(B,x)$
of $\mathcal{A}(B,x)$ associated with a slice
$\mathbf{i}=\mathbf{i}(0)\,|\, 
\cdots \,|\, \mathbf{i}(t-1)$
of a regular $\nu$-period of $B$}
is the  subring of $\mathcal{A}(B,x)$
generated by $[x_i(u)]_{\mathbf{1}}$
 ($(i,u)\in P_+$),
or equivalently, generated by
$[\tilde{x}_i(u)]_{\mathbf{1}}$ ($(i,u\in \tilde{P}_+$).
\par
(2) The {\em Y-subgroup $\mathcal{G}_{\mathbf{i}}(B,y)$
of $\mathcal{G}(B,y)$ associated with a 
slice
$\mathbf{i}=\mathbf{i}(0)\,|\, 
\cdots \,|\, \mathbf{i}(t-1)$
 of a regular $\nu$-period of $B$}
is the multiplicative subgroup of $\mathcal{G}(B,y)$
generated by
 $y_i(u)$ and $1+y_i(u)$ ($(i,u)\in P_+$).
\end{Definition}

The following fact was casually mentioned in Remark
\ref{rem:choice}.

\begin{Proposition}
\label{prop:choice}
The algebra $\mathcal{A}_{\mathbf{i}}(B,x)$
and the group $\mathcal{G}_{\mathbf{i}}(B,y)$
depend only on $\mathbf{i}$
and  do not depend on
the choice of a slice of $\mathbf{i}$.
\end{Proposition}
\begin{proof}
Suppose we take two different slices of
$\mathbf{i}$.
Then, there is a natural bijection
between the forward mutation points $(i,u) \leftrightarrow
(i',u')$ for two choices.
Then, we have $y_i(u)= y_{i'}(u')$ and
$x_i(u)= x_{i'}(u')$ in
$\mathcal{A}(B,x,y)$,
thanks to the condition \eqref{eq:compat} and
Lemma \ref{lem:order}.
\end{proof}

Next, we introduce the corresponding ring/group
for
 the standalone T-system/Y-system.

\begin{Definition}
\label{def:SA}
(1) Let $\EuScript{T}(B,\mathbf{i})$ be the commutative ring over
$\mathbb{Z}$ with identity element, with
generators 
 $T_i(u)^{\pm 1}$ ($(i,u)\in \tilde{P}_+$)
and relations \eqref{eq:ti},
where $[\tilde{x}_i(u)]_{\mathbf{1}}$ is replaced with $T_i(u)$,
together with $T_i(u) T_i(u)^{-1}=1$.
Let $\EuScript{T}^{\circ}(B,\mathbf{i})$ be the 
subring of $\EuScript{T}(B,\mathbf{i})$
generated by $T_i(u)$ ($(i,u)\in \tilde{P}_+$).
\par
(2) Let $\EuScript{Y}(B,\mathbf{i})$ be the semifield
with generators $Y_i(u)$ ($(i,u)\in P_+$)
and relations \eqref{eq:yi},
where $y_i(u)$ is replaced with $Y_i(u)$.
Let $\EuScript{Y}^{\circ}(B,\mathbf{i})$ be the multiplicative
subgroup of $\EuScript{Y}(B,\mathbf{i})$
generated by $Y_i(u)$ and $1+Y_i(u)$ ($(i,u)\in P_+$).
\end{Definition}

Two rings/groups defined above are isomorphic.

\begin{Theorem}
\label{thm:iso}
(1) The ring $\EuScript{T}^{\circ}(B,\mathbf{i})$ is isomorphic
to $\mathcal{A}_{\mathbf{i}}(B,x)$ by the correspondence
$T_i(u)\mapsto [\tilde{x}_i(u)]_{\mathbf{1}}$.
\par
(2) The group $\EuScript{Y}^{\circ}(B,\mathbf{i})$ is isomorphic
to $\mathcal{G}_{\mathbf{i}}(B,y)$ by the correspondence
$Y_i(u)\mapsto y_i(u)$ and 
$1+Y_i(u)\mapsto 1+y_i(u)$.
\end{Theorem}
\begin{proof}
(1)
The map $\rho :T_i(u)\mapsto [\tilde{x}_i(u)]_{\mathbf{1}}$
is a ring homomorphism by definition.
We can construct the inverse of $\rho$ as follows.
For each $i\in I$, let $u_i\in \mathbb{Z}$ be the smallest
nonnegative $u_i$ such that $(i,u_i)\in P_+$.
We define a ring homomorphism
$\phi:\mathbb{Z}[x_i^{\pm1}]_{i\in I}
\rightarrow \EuScript{T}(B,\mathbf{i})$ by
$x_i^{\pm1}\mapsto T_i(u_i-\frac{tg_i}{2})^{\pm1}$.
Thus, we have
$\phi:[\tilde{x}_i(u_i-\frac{tg_i}{2})]_{\mathbf{1}}
= [x_i(u_i)]_{\mathbf{1}}= [x_i(0)]_{\mathbf{1}}\mapsto
T_i(u_i-\frac{tg_i}{2})$.
Furthermore, one can prove
that 
$\phi:[\tilde{x}_i(u)]_{\mathbf{1}}\mapsto T_i(u)$
for any $(i,u)\in \tilde{P}_+$
by induction on the forward mutations
for  $u>u_i-\frac{tg_i}{2}$ and on the backward mutations for $u<u_i-\frac{tg_i}{2}$,
using the common T-systems for the both sides.
By the restriction of $\phi$
to $\mathcal{A}_{\mathbf{i}}(B,x)$,
we obtain a ring homomorphism $\varphi:
\mathcal{A}_{\mathbf{i}}(B,x)\rightarrow 
\EuScript{T}^{\circ}(B,\mathbf{i})$,
which is the inverse of $\rho$.
\par
(2) This is parallel to (1).
The map $\rho :Y_i(u)\mapsto y_i(u)$,
$1+Y_i(u)\mapsto 1+y_i(u)$
is a group homomorphism by definition.
We can construct the inverse of $\rho$ as follows.
For each $i\in I$, let $u_i\in \mathbb{Z}$ be the largest
nonpositive $u_i$ such that $(i,u_i)\in P_+$.
We define a semifield homomorphism
$\phi:\mathbb{P}_{\mathrm{univ}}(y_i)_{i\in I}
\rightarrow \EuScript{Y}(B,\mathbf{i})$ as follows.
If $u_i=0$, then $\phi(y_i)=Y_i(0)$.
If $u_i < 0$, we define
\begin{align}
\phi(y_i)=
Y_i(u_i)^{-1}
\frac{
\displaystyle
\prod_{(j,v)} (1+Y_j(v))^{-b_{ji}(v)}
}
{
\displaystyle
\prod_{(j,v)} (1+Y_j(v)^{-1})^{b_{ji}(v)},
}
\end{align}
where
 the product in the numerator is taken
for $(j,v)\in P_+$ such that
$u_i < v < 0$ and $b_{ji}(v)<0$,
and 
 the product in the denominator is taken
for $(j,v)\in P_+$ such that
$u_i < v < 0$ and $b_{ji}(v)>0$.
Then, we have $\phi:y_i(u_i)\mapsto Y_i(u_i)$.
Furthermore,
 one can prove
that 
$\phi:y_i(u)\mapsto Y_i(u)$
for any $(i,u)\in P_+$
by induction on the forward mutations
for  $u>u_i$ and on the backward mutations for $u<u_i$,
using the common Y-systems for the both sides.
By the restriction of $\phi$
to $\mathcal{G}_{\mathbf{i}}(B,y)$,
we obtain a group homomorphism $\varphi:
\mathcal{G}_{\mathbf{i}}(B,y)\rightarrow 
\EuScript{Y}^{\circ}(B,\mathbf{i})$,
which is the inverse of $\rho$.
\end{proof}

Aside from the direct connection
between T- and Y-systems in Proposition
\ref{prop:TY1},
there is an algebraic connection,
which has been noticed since the inception
of the original T- and Y-systems \cite{Klumper92,Kuniba94a}.

\begin{Proposition}
\label{prop:GH}
Let $\EuScript{T}(B,\mathbf{i})$ be the ring in Definition
\ref{def:SA}.
For each $(i,u)\in P_+$, we set 
\begin{align}
\label{eq:YT1}
Y_i(u):=\frac{
\displaystyle
\prod_{(j,v)\in \tilde{P}_+}
T_j(v)^{H_+(j,v;i,u)}
}
{
\displaystyle
\prod_{(j,v)\in\tilde{P}_+}
T_j(v)^{H_-(j,v;i,u)}
}.
\end{align}
Then, $Y_i(u)$ satisfies the Y-system
\eqref{eq:yi} in $\EuScript{T}(B,\mathbf{i})$
by replacing $y_i(u)$ with $Y_i(u)$.
\end{Proposition}
Note that \eqref{eq:YT1} is the ratio of
the first and second terms in \eqref{eq:ti}.

\begin{proof}
Thanks to the isomorphism in Theorem
\ref{thm:iso},
one can work in the localization
of $\mathcal{A}_{\mathbf{i}}(B,x)$
by generators $[\tilde{x}_i(u)]_{\mathbf{1}}$
 ($(i,u)\in\tilde{P}_+$),
which is a subring
of the ambient field
$\mathbb{Q}(x)$.
The claim is translated therein as follows:
{\em For each $(i,u)\in P_+$, we set
\begin{align}
\bar{y}_i(u)&:=
\frac{
\displaystyle
\prod_{(j,v)\in\tilde{P}_+}
[\tilde{x}_j(v)]_{\mathbf{1}}^{H_+(j,v;i,u)}
}
{
\displaystyle
\prod_{(j,v)\in\tilde{P}_+}
[\tilde{x}_j(v)]_{\mathbf{1}}^{H_-(j,v;i,u)}
}
=
\frac{
\displaystyle
\prod_{(j,v)\in P_+}
[{x}_j(v)]_{\mathbf{1}}^{{H}'_+(j,v;i,u)}
}
{
\displaystyle
\prod_{(j,v)\in P_+}
[{x}_j(v)]_{\mathbf{1}}^{{H}'_-(j,v;i,u)}
}
=\prod_{j\in I} [x_j(u)]_{\mathbf{1}}^{b_{ji}(u)}.
\end{align}
Then $\bar{y}_i(u)$ satisfies the Y-system
 \eqref{eq:yi}
by replacing $y_i(u)$ with $\bar{y}_i(u)$.}
In fact, this claim  is an immediate consequence of
 \cite[Proposition 3.9]{Fomin07}.
\end{proof}

\begin{Remark}
Some classic examples of T- and Y-systems
are not always in the `straight form' presented here,
but represented by
generators $T_{\bar{i}}(u)$
and $Y_{\bar{i}}(u)$
whose indices $\bar{i}$ belong to
 the orbit space $I/\nu$ of $I$ by $\nu$.
For example, the T- and Y-systems for type $(X,\ell)=
(B_4,4)$ in  Section
\ref{subsec:exTY}
\cite{Inoue10a},
and the sine-Gordon T- and Y-systems
for Example \ref{ex:sine} \cite{Nakanishi10b}
are such cases.
In these examples, it is  just  a `change of
notation' for generators.
However, this makes the reconstruction of the initial
exchange matrix $B$ from given T- or Y-systems
nontrivial,  because {\em a priori\/}
we only know $I/\nu$,
and  we have to find out true index set $I$ and $\nu$
with some guesswork.
\end{Remark}

\subsection{T- and Y-systems for general period}

\label{subsec:nonadm}
Conceptually, the notions of T- and Y-systems
can be straightforwardly extended to  general $\nu$-periods of $B$,
though they become a little apart from the
`classic' T- and Y-systems.
We will use them  in Section \ref{sec:dilog}.

Let  $\mathbf{i}=\mathbf{i}(0)\,|\, 
\cdots \,|\, \mathbf{i}(t-1)$ be a slice of 
any (not necessarily regular) $\nu$-period $\mathbf{i}$ of $B$.
One can still define the sequence of
 seeds $(B(u),x(u),y(u))$ ($u\in \mathbb{Z}$)
 and the forward mutation points $(i,u)\in P_+$
as in the regular case.

Fix $i\in I$, and let
\begin{align}
\label{eq:for1}
\dots ,(i,u),\, (i,u'),\,  (i,u''),\, \dots
\quad
(\dots < u < u' < u''<\dots)
\end{align}
be the sequence of the forward mutation points.
(It may be empty for some $i$.)
In general,
if it is not empty,
the sequence  $\dots, u,  u', u'',\dots$
is periodic for $u \rightarrow u + tg$,
but it does  {\em not necessarily\/} have the common difference.
For each $(i,u)\in P_+$,
 let $(i,u+\lambda_+(i,u))$
and $(i,u-\lambda_-(i,u))$  be the nearest ones
to $(i,u)$ in the sequence \eqref{eq:for1}
in  the forward and backward directions, respectively;
in other words,
$(i,u-\lambda_-(i,u))$, $(i,u)$, $(i,u+\lambda_+(i,u))$
are three consecutive forward mutation points in \eqref{eq:for1}.
If $\mathbf{i}$ is regular,
then $\lambda_{\pm}(i,u)=tg_i$,
which is the common difference
(therefore, called regular).
In general, we have $0< \lambda_{\pm}(i,u)<tg$,
$\lambda_{\pm}(i,u+tg)=\lambda_{\pm}(i,u)$,
and
\begin{align}
\lambda_+(i,u-\lambda_-(i,u))=\lambda_-(i,u).
\end{align}

Let $J(\mathbf{i},\nu)$ be the subset of $I$
consisting of all the components of $\mathbf{j}(\mathbf{i},\nu)$.
Note that the condition (A1) in Definition \ref{def:regular}
means that $J(\mathbf{i},\nu)=I$.

Using these notations,
the relations \eqref{eq:yi'} and \eqref{eq:xi'}
are  generalized as follows.
For  $(i,u)\in P_+$,
\begin{align}
\label{eq:yi'2}
\textstyle
y_i\left(u \right)y_i\left(u+\lambda_+(i,u) \right)
&=
\frac{
\displaystyle
\prod_{(j,v)\in P_+} (1+y_j(v))^{G'_+(j,v;i,u)}
}
{
\displaystyle
\prod_{(j,v)\in P_+} (1+y_j(v)^{-1})^{G'_-(j,v;i,u)}
},\\
\label{eq:G'2}
G'_{\pm}(j,v;i,u)&=
\begin{cases}
\mp b_{ji}(v) & v\in (u, u+\lambda_+(i,u)),
b_{ji}(v)\lessgtr 0\\
0 & \mbox{otherwise},
\end{cases}
\end{align}
and
\begin{align}
\label{eq:xi'3}
\begin{split}
x_i(u)x_i(u+\lambda_+(i,u))
&=
\frac{y_i(u)}{1+y_i(u)}
\prod_{j\in I\setminus J(\mathbf{i},\nu):\, b_{ji}>0}
 x_j^{b_{ji}(u)}
\prod_{(j,v)\in P_+} x_j(v)^{H'_+(j,v;i,u)}\\
&\quad
+
\frac{1}{1+y_i(u)}
\prod_{j\in I\setminus J(\mathbf{i},\nu):\, b_{ji}<0}
 x_j^{-b_{ji}(u)}
\prod_{(j,v)\in P_+} x_j(v)^{H'_-(j,v;i,u)},
\end{split}\\
\label{eq:tH2}
H'_{\pm}(j,v;i,u)&=
\begin{cases}
\pm b_{ji}(u)&
u\in (v- \lambda_-(j,v),v), b_{ji}(u)\gtrless 0\\
0 & \mbox{otherwise}.
\end{cases}
\end{align}
When a $\nu$-period $\mathbf{i}$ of $B$ 
satisfies the condition (A1) in Definition \ref{def:regular},
the relation \eqref{eq:xi'3} slightly simplifies as 
\begin{align}
\label{eq:xi'2}
\begin{split}
x_i(u)x_i(u+\lambda_+(i,u))
&=
\frac{y_i(u)}{1+y_i(u)}
\prod_{(j,v)\in P_+} x_j(v)^{H'_+(j,v;i,u)}\\
&\quad
+
\frac{1}{1+y_i(u)}
\prod_{(j,v)\in P_+} x_j(v)^{H'_-(j,v;i,u)}.
\end{split}
\end{align}

Unfortunately,
one cannot simultaneously rewrite \eqref{eq:yi'2}
and \eqref{eq:xi'3}/\eqref{eq:xi'2} into
the balanced form similar to \eqref{eq:yi}
and \eqref{eq:xi}.
Therefore, we just regard
\eqref{eq:yi'2} and \eqref{eq:xi'3}/\eqref{eq:xi'2} as
the Y- and T-systems associated with
a slice $\mathbf{i}=\mathbf{i}(0)\,|\, 
\cdots \,|\, \mathbf{i}(t-1)$
 of a $\nu$-period of $B$.

Accordingly, Proposition \ref{prop:TY1} is generalized as follows.
\begin{Proposition}[Duality]
\label{prop:TY1'}
Let $D=\mathrm{diag} (d_i)_{i\in I}$ be
a diagonal matrix such that ${}^{t}(DB)=-DB$.
Let $G'_{\pm}(j,v;i,u)$ and 
$H'_{\pm}(j,v;i,u)$ the ones
in \eqref{eq:G'2}
and \eqref{eq:tH2}, respectively.
Then, the following relation holds
for any $(j,v)\in P_+$ and $(i,u)\in \tilde{P}_+$.
\begin{align}
\label{eq:GH1'}
d_j G'_{\pm} (j,v;i,u-\lambda_-(i,u))= d_i H'_{\pm} (i,u;j,v).
\end{align}
In particular, if $B$ is skew symmetric, we have
\begin{align}
\label{eq:GH2'}
 G'_{\pm} (j,v;i,u-\lambda_-(i,u))= H'_{\pm} (i,u;j,v).
\end{align}
\end{Proposition}
\begin{proof}
This is immediate from \eqref{eq:G'2} and \eqref{eq:tH2}.
\end{proof}
The results in Section \ref{subsec:stand} can be
also generalized and/or modified straightforwardly.
We leave it as an exercise for the reader.

\section{Dilogarithm identities}
\label{sec:dilog}

Now we are ready to work on the main subject
of the paper, the dilogarithm identities
associated with any period of a seed.
They are natural generalizations of the results
in the former examples
\cite{Frenkel95,Chapoton05, Nakanishi09,Inoue10a,Inoue10b,Nakanishi10b}.
The subject has originated from the pioneering works
by Bazhanov, Kirillov, and Reshetikhin \cite{Kirillov86,
Kirillov89,Kirillov90,
Bazhanov90}.
See also \cite{Kirillov95,Zagier07, Nahm07} for further background
of the identities.

\subsection{Dilogarithm identities}
Here we concentrate on the case when the exchange matrix
 $B$ is {\em skew symmetric\/}.
(See Section \ref{subsec:dilog2} for the skew symmetrizable case.)
Let $\mathbf{i}$ be any period of $(B,x,y)$.
In particular,
$\mathbf{i}$ is also a period of $B$.
Let $\mathbf{i}=\mathbf{i}(0)\,|\, 
\cdots \,|\, \mathbf{i}(t-1)$
be  any slice of $\mathbf{i}$.
Then, the sequence of seeds $(B(u),x(u),y(u))$ ($u\in \mathbb{Z}$)
and the forward
mutation points $(i,u)\in P_+$ are defined 
as in Section \ref{sec:T},
and we have the associated Y-system
\eqref{eq:yi'2}.

To emphasize the periodicity, we set $\Omega:=t$.
Then, by the periodicity assumption, we have
\begin{align}
\label{eq:yperiod}
y_i(u+\Omega)&=y_i(u).
\end{align}
Accordingly, we define the `fundamental regions'
$S_+$  of forward mutation points by
\begin{align}
S_+&=\{ (i,u)\in P_+ \mid 0\leq u < \Omega\}.
\end{align}

Thanks to Theorem \ref{thm:pos}, each Laurent monomial
$[y_i(u)]_{\mathbf{T}}$ in $y$ ($(i,u)\in P_+$) is
either positive or negative.
Let $N_+$ (resp. $N_-$) be the total number of
positive (resp. negative) monomials $[y_i(u)]_{\mathbf{T}}$
in the fundamental region $S_+$.

Let $L(x)$ be the Rogers dilogarithm function  \cite{Lewin81}.
\begin{align}
\label{eq:L0}
L(x)=-\frac{1}{2}\int_{0}^x 
\left\{ \frac{\log(1-y)}{y}+
\frac{\log y}{1-y}
\right\} dy
\quad (0\leq x\leq 1).
\end{align}
It satisfies the following properties.
\begin{gather}
\label{eq:L01}
L(0)=0,
\quad L(1)=\frac{\pi^2}{6},\\
\label{eq:euler}
L(x)+L(1-x)=\frac{\pi^2}{6}
\quad (0\leq x\leq 1).
\end{gather}

The following is the main result of the paper.
\begin{Theorem}[Dilogarithm identity in the skew symmetric case]
\label{thm:DI}
Suppose that
a family of real positive numbers
$\{Y_i(u) \mid (i,u)\in P_+\}$
satisfies the Y-system \eqref{eq:yi'2}
by replacing $y_i(u)$ with $Y_i(u)$.
Then, we have the identities.
\begin{align}
\label{eq:DI}
\frac{6}{\pi^2}
\sum_{
(i,u)\in S_+
}
L\left(
\frac{Y_i(u)}{1+Y_i(u)}
\right)
&=N_-,\\
\label{eq:DI'}
\frac{6}{\pi^2}
\sum_{
(i,u)\in S_+
}
L\left(
\frac{1}{1+Y_i(u)}
\right)
&=N_+.
\end{align}
\end{Theorem}

Two identities
\eqref{eq:DI} and \eqref{eq:DI'}
are equivalent to each other due to \eqref{eq:euler}.

\subsection{Proof of Theorem \ref{thm:DI}}

We prove Theorem \ref{thm:DI}
using the method developed by
 \cite{Frenkel95,Chapoton05,Nakanishi09}.

We start from a  very general theorem
on the `constancy property' of dilogarithm sum
by Frenkel-Szenes \cite{Frenkel95}.
For any multiplicative abelian group $A$,
let $A\otimes_{\mathbb{Z}} A$ be the additive
abelian group with generators $f\otimes g$
($f,g\in A$) and relations
\begin{align}
(fg)\otimes h = f\otimes h + g \otimes h,
\quad f\otimes (gh)=f\otimes g + f\otimes h.
\end{align}
It follows that we have
\begin{align}
1\otimes h = h\otimes 1=0,
\quad f^{-1}\otimes g= f\otimes g^{-1}=- f\otimes g.
\end{align}
Let $S^2 A$ be the subgroup of
$A\otimes_{\mathbb{Z}} A$ generated 
by $f\otimes f$ ($f\in A$).
Define $\bigwedge^2 A$ be the quotient group
of $A\otimes_{\mathbb{Z}} A$ by $S^2 A$.
In $A\otimes_{\mathbb{Z}} A$
we use $\wedge$ instead of  $\otimes$.
Hence, $f\wedge g = -g \wedge f$ holds.

Now let $\mathcal{I}$ be any open or closed interval of $\mathbb{R}$,
and  let $\EuScript{C}(\mathcal{I})$ be the multiplicative abelian group
of all the differentiable functions $f(t)$
from $\mathcal{I}$ to the set of all the positive real numbers
$\mathbb{R}_{+}$.

\begin{Theorem}[{\cite[Proposition 1]{Frenkel95}}]
\label{thm:const}
Let  $f_1(t),\dots,f_k(t)$ be differentiable functions
from $\mathcal{I}$ to $(0,1)$.
Suppose that they satisfy the following
relation in $\bigwedge^2 \EuScript{C}(\mathcal{I})$.
\begin{align}
\label{eq:const1}
\sum_{i=1}^k f_i(t) \wedge (1-f_i(t)) = 0
\quad \mbox{\rm (constancy condition)}.
\end{align}
Then, the dilogarithm sum $\sum_{i=1}^k L(f_i(t))$ is constant
with respect to $t\in \mathcal{I}$.
\end{Theorem}
We remark that the proof of the theorem by \cite{Frenkel95} is quite simple,
just showing the derivative of the dilogarithm sum vanishes
by the symmetry reason.

Next we use the idea of Chapoton \cite{Chapoton05}
to integrate
the condition \eqref{eq:const1} in the cluster algebra setting.
We are going to prove the following claim.

\begin{Proposition}
\label{prop:const}
In $\bigwedge ^2 \mathbb{P}_{\mathrm{univ}}(y)$,
the following relation holds.
\begin{align}
\label{eq:const2}
\sum_{(i,u)\in S_+}
\frac{y_i(u)}{1+y_i(u)}\wedge \frac{1}{1+y_i(u)}=0,
\end{align}
or, equivalently,
\begin{align}
\label{eq:const3}
\sum_{(i,u)\in S_+}
y_i(u)\wedge (1+y_i(u))=0.
\end{align}
\end{Proposition}
It is clear that the relations \eqref{eq:const2}
and \eqref{eq:const3} are equivalent. 
We call them
the {\em constancy condition\/}
for the dilogarithm identities \eqref{eq:DI} and \eqref{eq:DI'}.
Temporarily assuming Proposition \ref{prop:const},
we prove the following theorem,
which is equivalent to Theorem \ref{thm:DI}.

\begin{Theorem}
Let $\mathbb{R}_+$ be the semifield of all the positive real numbers
with the usual multiplication and addition.
Then, for any semifield homomorphism $\varphi:
\mathbb{P}_{\mathrm{univ}}(y) \rightarrow \mathbb{R}_+$,
we have the identity,
\begin{align}
\label{eq:DI1}
\frac{6}{\pi^2}
\sum_{
(i,u)\in S_+
}
L\left(
\varphi
\left(
\frac{y_i(u)}{1+y_i(u)}
\right)
\right)
&=N_-,\\
\label{eq:DI'1}
\frac{6}{\pi^2}
\sum_{
(i,u)\in S_+
}
L\left(
\varphi
\left(
\frac{1}{1+y_i(u)}
\right)
\right)
&=N_+.
\end{align}
\end{Theorem}
\begin{proof}
We concentrate on the first case \eqref{eq:DI1}.
The proof separates into two steps.

{\em Step 1: Consistency.}
Let us first show that the left hand side of \eqref{eq:DI1}
is independent of the choice of $\varphi$.
Note that any semifield homomorphism
$\varphi$ is determined by the
values $\varphi(y_i)$ ($i\in I$) of generators
of $\mathbb{P}_{\mathrm{univ}}(y)$.
Let $\varphi_0$ and
$\varphi_1$ be different homomorphisms from $\mathbb{P}_{\mathrm{univ}}(y)$
to $\mathbb{R}_+$.
Then we introduce the one parameter family of homomorphisms
$\varphi_t: \mathbb{P}_{\mathrm{univ}}(y)\rightarrow \mathbb{R}_+$
 ($t\in [0,1]$) interpolating $\varphi_0$ to $\varphi_1$
by $\varphi_t(y_i)= (1-t) \varphi_0(y_i)+ t \varphi_1(y_i)$.
Set
$f_{i,u}(t):= \varphi_t(y_i(u)/(1+y_i(u)))$ ($(i,u)\in S_+$)
and regard them as a family of functions of $t$ 
in the interval  $[0,1]$.
Then, due to \eqref{eq:const2},
they satisfy the constancy condition \eqref{eq:const1}
with $\mathcal{I}=[0,1]$.
Then, by Theorem \ref{thm:const},
the left hand side of \eqref{eq:DI1} for $\varphi=\varphi_t$
is independent of $t$; in particular, it is the same
for $\varphi=\varphi_0$ and $\varphi_1$.
\par
{\em Step 2: Evaluation at $0/\infty$ limit.}
Since we established the constancy of the left hand side
of \eqref{eq:DI1}, we evaluate it at a certain limit of $\varphi$
where {\em each value $\varphi(y_i(u)/(1+y_i(u)))$ $((i,u)\in S_+)$
goes to either 0 or $\infty$} ($0/\infty$ limit).
Then, by \eqref{eq:L01}, the value of the left hand side
of \eqref{eq:DI1} is the total
number of $(i,u)\in S_+$ such that
the value $\varphi(y_i(u)/(1+y_i(u)))$ goes to $\infty$.
It is not obvious that such a limit exists.
However, simply
take the one parameter family of the
homomorphism $\varphi_t:
\mathbb{P}_{\mathrm{univ}}\rightarrow \mathbb{R}_+$
 ($t\in (0,\varepsilon)$) for some $\varepsilon>0$
 defined by $\varphi_t(y_i)=t$ for any $i\in I$.
Then, thanks to \eqref{eq:Yfact}
and the parts (a) and (c) of Conjecture \ref{conj:pos}/Theorem \ref{thm:pos},
the limit $\lim_{t\rightarrow 0} \varphi_t$ is indeed 
a $0/\infty$  limit,
and  the total
number of $(i,u)\in S_+$ such that
the value $\varphi(y_i(u)/(1+y_i(u)))$ goes to $\infty$
is exactly $N_-$.
\end{proof}

\begin{Remark}
In the above proof, Step 1 (constancy) is
due to \cite{Frenkel95},
and Step 2 (evaluation at $0/\infty$ limit)
is due to \cite{Chapoton05}.
\end{Remark}

\subsection{First proof of Proposition \ref{prop:const}}
Now we only have  to prove Proposition \ref{prop:const}.
We give two proofs in this and the next subsections.
The first one here is
a generalization of the
former proofs for special cases \cite{Nakanishi09,Inoue10a},
which is a somewhat brute force proof with
 `change of indices'.

To start, we claim that one can assume
all the components of $\mathbf{i}$ exhaust $I$.
In fact, if the condition does not hold,
one may reduce the index set $I$ so that the condition holds,
since doing that does not affect the  Y-system \eqref{eq:yi'2}.
Note that the condition is nothing but
the condition (A1) in Definition \ref{def:regular} with
$\nu=\mathrm{id}$.
Thus, the accompanying T-system 
has the simplified form \eqref{eq:xi'2}.
Furthermore, by the periodicity assumption
and Theorem \ref{thm:r/e} (a),
we have
\begin{align}
\label{eq:xperiod}
x_i(u+\Omega)&=x_i(u).
\end{align}

Now, let $F_i(u)$ 
 be the $F$-polynomials at $(i,u)$.

\begin{Lemma}
\label{lem:F}
The following properties hold.

\par
(a) Periodicity: $F_i(u+\Omega)=F_i(u)$.

\par
(b) For $(i,u)\in {S}_+$,
\begin{align}
\label{eq:Fi}
\begin{split}
\textstyle
F_i(u)F_i(u+\lambda_+(i,u))
&=
\left[
\frac{y_i(u)}{1+y_i(u)}
\right]_{\mathbf{T}}
\prod_{(j,v)\in P_+} F_j(v)^{H'_+(j,v;i,u)}\\
&\quad
+
\left[
\frac{1}{1+y_i(u)}
\right]_{\mathbf{T}}
\prod_{(j,v)\in P_+} F_j(v)^{H'_-(j,v;i,u)}.
\end{split}
\end{align}
\par
(c) For $(i,u)\in{S}_+$, 
\begin{align}
\label{eq:yi1}
y_i(u)&= [y_i(u)]_{\mathbf{T}}
\frac
{
\displaystyle
\prod_{(j,v)\in P_+} F_j(v)^{H'_+(j,v;i,u)}
}
{
\displaystyle
\prod_{(j,v)\in P_+} F_j(v)^{H'_-(j,v;i,u)}
},\\
\label{eq:yi2}
1+y_i(u)&= [1+y_i(u)]_{\mathbf{T}}
\frac
{
F_i(u)F_i(u+\lambda_+(i,u))
}
{
\displaystyle
\prod_{(j,v)\in P_+} F_j(v)^{H'_-(j,v;i,u)}
}.
\end{align}
\end{Lemma}
\begin{proof}
(a) This is obtained from the specialization of \eqref{eq:xperiod}.
(b) This is obtained from the specialization of \eqref{eq:xi'2}.
(c) The first equality is obtained by rewriting
\eqref{eq:Yfact} with \eqref{eq:tH2}.
The second one is obtained from the first one
and (b).
\end{proof}

\begin{proof}[Proof of Proposition \ref{prop:const}]
We prove \eqref{eq:const3}.
We put \eqref{eq:yi1} and \eqref{eq:yi2} into
the left hand side of \eqref{eq:const3},
expand it, then, sum them up into three parts as follows.

The first part consists of the terms only involving tropical
coefficients, i.e.,
\begin{align}
\label{eq:const4}
\sum_{(i,u)\in S_+}
[y_i(u)]_{\mathbf{T}}\wedge [1+y_i(u)]_{\mathbf{T}}.
\end{align}
By Theorem \ref{thm:pos}, each monomial
$[y_i(u)]_{\mathbf{T}}$ is either positive or negative.
If it is positive, then
$[y_i(u)]_{\mathbf{T}}\wedge [1+y_i(u)]_{\mathbf{T}}=
[y_i(u)]_{\mathbf{T}}\wedge 1=0$.
If it is negative, then,
$[y_i(u)]_{\mathbf{T}}\wedge [1+y_i(u)]_{\mathbf{T}}=
[y_i(u)]_{\mathbf{T}}\wedge [y_i(u)]_{\mathbf{T}}=0$.
Therefore, the sum \eqref{eq:const4} vanishes.

The second part consists of the terms involving both  tropical
coefficients and $F$-polynomials.
We separate them into five parts,
\begin{align}
\label{eq:yF}
\begin{split}
\sum_{(i,u)\in S_+}
[y_{i}(u)]_{\mathrm{T}}
\wedge
\textstyle
F_{i}(u),
&
%=
%\sum_{(j,v)\in S_+}
%\textstyle
%[y_{j}(v)]_{\mathrm{T}}
%\wedge
%F_{j}(v),
\\
\sum_{(i,u)\in S_+}
\textstyle
[y_{i}(u)]_{\mathrm{T}}
\wedge
F_{i}(u+\lambda_+(i,u))
&=
\sum_{(i,u)\in S_+}
\textstyle
[y_{i}(u-\lambda_-(i,u))]_{\mathrm{T}}
\wedge
F_{i}(u),
\\ 
-\sum_{(i,u)\in S_+}
[y_{i}(u)]_{\mathrm{T}}
\wedge
\prod_{(j,v)\in P_+}
&
F_{j}(v)^{H'_-(j,v;i,u)}\\
&=
\sum_{(i,u)\in S_+}
\prod_{(j,v)\in P_+}
[y_j(v)]_{\mathrm{T}}^{-H'_-(i,u;j,v)}
\wedge
F_{i}(u),
\\
-\sum_{(i,u)\in S_+}
[1+y_{i}(u)]_{\mathrm{T}}
\wedge
\prod_{(j,v)\in P_+
}
&
F_{j}(v)^{H'_+(j,v;i,u)}\\
&=
\sum_{(i,u)\in S_+}
\prod_{(j,v)\in P_+}
[1+y_j(v)]_{\mathrm{T}}^{-H'_+(i,u;j,v)}
\wedge
F_{i}(u),
\\
\sum_{(i,u)\in S_+}
[1+y_{i}(u)]_{\mathrm{T}}
\wedge
\prod_{(j,v)\in P_+}
&
F_{j}(v)^{H'_-(j,v;i,u)}\\
&=
\sum_{(i,u)\in S_+}
\prod_{(j,v)\in P_+}
[1+y_{j}(v)]_{\mathrm{T}}^{H'_-(i,u;j,v)}
\wedge
F_{i}(u),
\end{split}
\end{align}
where we changed  indices and
also
used  the periodicity of $[y_i(u)]_{\mathbf{T}}$,
$F_i(u)$, $\lambda_-(i,u)$, and $H'_{\pm}(j,v;i,u)$.
Recall the relation $H'_{\pm}(i,u;j,v)=
G'_{\pm}(j,v;i,u-\lambda_-(i,u))$ in \eqref{eq:GH2'}.
Then, the sum of the above five terms vanishes due to
the  `tropical Y-system'
\begin{align}
\textstyle
[y_i\left(u-\lambda_-(i,u)\right)]_{\mathbf{T}}
[y_i\left(u\right)]_{\mathbf{T}}
&=
\frac{
\displaystyle
\prod_{(j,v)\in P_+} [1+y_j(v)]_{\mathbf{T}}^{G'_+(j,v;i,u-\lambda_-(i,u))}
}
{
\displaystyle
\prod_{(j,v)\in P_+} [1+y_j(v)^{-1}]_{\mathbf{T}}^{G'_-(j,v;i,u-\lambda_-(i,u))}
},
\end{align}
which is a specialization of  \eqref{eq:yi'2}.
\par
The third part consists of the terms involving only
$F$-polynomials. 
It turns out that
this part requires the most elaborated treatment.
We separate them into three parts,
\begin{align}
\label{eq:FF1}
\mathrm{(A)}&=\sum_{(i,u)\in S_+} 
\frac
{
\displaystyle
\prod_{(j,v)\in P_+} F_j(v)^{H'_+(j,v;i,u)}
}
{
\displaystyle
\prod_{(j,v)\in P_+} F_j(v)^{H'_-(j,v;i,u)}
}
\wedge
\textstyle
F_i(u),\\
\label{eq:FF2}
\mathrm{(B)}&=\sum_{(i,u)\in S_+} 
\frac
{
\displaystyle
\prod_{(j,v)\in P_+} F_j(v)^{H'_+(j,v;i,u)}
}
{
\displaystyle
\prod_{(j,v)\in P_+} F_j(v)^{H'_-(j,v;i,u)}
}
\wedge
\textstyle
F_i(u+\lambda_+(i,u)),\\
\label{eq:FF3}
\mathrm{(C)}&=\sum_{(i,u)\in S_+} 
\displaystyle
\prod_{(j,v)\in P_+} F_j(v)^{H'_+(j,v;i,u)}
\wedge
\prod_{(j,v)\in P_+} F_j(v)^{-H'_-(j,v;i,u)}.
\end{align}
Let us rewrite each term so that their cancellation becomes manifest.

The first term (A) is rewritten as follows.
\begin{align*}
\mathrm{(A)}&=\sum_{(i,u)\in S_+} 
\prod_{\scriptstyle (j,v)\in P_+ \atop
\scriptstyle
u\in (v-\lambda_-(j,v),v)
} 
F_j(v)^{b_{ji}(u)}
\wedge
\textstyle
F_i(u)\\
&=\sum_{\scriptstyle (i,u)\in S_+,\, (j,v)\in P_+ \atop
\scriptstyle
u\in (v-\lambda_-(j,v),v)
} 
\textstyle
b_{ji}(u)
F_j(v)
\wedge
\textstyle
F_i(u)\\
&=\sum_{\scriptstyle (i,u)\in S_+,\, (j,v)\in P_+ \atop
\scriptstyle
v\in (u-\lambda_-(i,u),u)
} 
\textstyle
b_{ji}(v)
F_j(v)
\wedge
\textstyle
F_i(u)\\
&
=
\frac{1}{2}
\sum_{\scriptstyle (i,u)\in S_+,\, (j,v)\in P_+ \atop
\scriptstyle
(u-\lambda_-(i,u),u)
\cap(v-\lambda_-(j,v),v)\neq \emptyset
} 
\textstyle
b_{ji}(\min(u,v))
F_j(v)
\wedge
\textstyle
F_i(u).
\end{align*}
Here the third line is obtained from
the second one by the exchange $(i,u)\leftrightarrow (j,v)$ of indices,
the skew symmetric property $b_{ji}(u)=-b_{ij}(u)$, and the periodicity;
the last line is obtained by averaging the
second and the third ones; therefore, there is
the factor $1/2$ in the front.
We also mention that, in the last line,
the pair $(i,u), (j,v)$ with
$u=v$ does not contribute
to the sum, because  in that case we have
$b_{ji}(u)=0$ due to the condition \eqref{eq:compat}.

Similarly, the second term (B) is rewritten as follows.
\begin{align*}
\mathrm{(B)}&=\sum_{(i,u)\in S_+} 
\prod_{\scriptstyle (j,v)\in P_+ \atop
\scriptstyle
u\in (v-\lambda_-(j,v),v)
} 
F_j(v)^{b_{ji}(u)}
\wedge
\textstyle
F_i(u+\lambda_+(i,u))\\
&=\sum_{\scriptstyle (i,u)\in S_+,\, (j,v)\in P_+ \atop
\scriptstyle
u-\lambda_-(i,u)\in (v-\lambda_-(j,v),v)
} 
\textstyle
b_{ji}(u-\lambda_-(i,u))
F_j(v)
\wedge
\textstyle
F_i(u)\\
&
=
\frac{1}{2}
\sum_{\scriptstyle (i,u)\in S_+,\, (j,v)\in P_+ \atop
\scriptstyle
(u-\lambda_-(i,u),u)
\cap(v-\lambda_-(j,v),v)\neq \emptyset
} 
\textstyle
b_{ji}(\max(u-\lambda_-(i,u), v-\lambda_-(j,v)))
F_j(v)
\wedge
\textstyle
F_i(u).
\end{align*}

The third term (C) is written as follows.
\begin{align*}
\mathrm{(C)}&=\sum_{(i,u),(j,v)\in P_+} 
\Biggl(
\sum_{\scriptstyle (k,w)\in S_+ \atop
{
\scriptstyle
k\in (u-\lambda_-(i,u),u)\cap
 (v-\lambda_-(j,v),v)\atop
\scriptstyle
b_{jk}(w)>0,\, b_{ik}(w)<0
}
} 
b_{jk}(w) b_{ik}(w)
\Biggr)
F_j(v)
\wedge
\textstyle
F_i(u)\\
&=\sum_{(i,u),(j,v)\in P_+} 
\Biggl(
\sum_{\scriptstyle (k,w)\in S_+ \atop
{
\scriptstyle
k\in (u-\lambda_-(i,u),u)\cap
 (v-\lambda_-(j,v),v)\atop
\scriptstyle
b_{jk}(w)<0,\, b_{ik}(w)>0
}
} 
-b_{jk}(w) b_{ik}(w)
\Biggr)
F_j(v)
\wedge
\textstyle
F_i(u)\\
&=\frac{1}{2}\sum_{(i,u)\in S_+, (j,v)\in P_+} 
\Biggl(
\sum_{\scriptstyle (k,w)\in P_+ \atop
\scriptstyle
k\in (u-\lambda_-(i,u),u)\cap
 (v-\lambda_-(j,v),v)
}\\
&\hskip100pt 
-\frac{1}{2}\Bigl( b_{jk}(w) |b_{ki}(w)|
+ |b_{jk}(w)|b_{ki}(w)\Bigr)
\Biggr)
F_j(v)
\wedge
\textstyle
F_i(u).
\end{align*}
Here the second line is obtained from
the first one by the exchange $(i,u)\leftrightarrow (j,v)$ of indices;
the last line is obtained by averaging the
first and the second ones; therefore, there is
the factor $1/2$ in the front.
We also mention that,
 in the last line,
$(k,w)$ with
 $b_{jk}(w)$ and  $b_{ik}(w)$ having  the same sign
 does not contribute
to the second sum,
 because in that case $b_{jk}(w)|b_{ki}(w)|+|b_{jk}(w)|b_{ki}(w)=0$.

Now the sum  $\mathrm{(A)}+\mathrm{(B)}+\mathrm{(C)}$ cancel
due to the mutation of matrices (see \eqref{eq:Bmut})
\begin{align}
\begin{split}
&\textstyle
b_{ji}(\min(u,v))
= -b_{ji}(\max(u-\lambda_i(i,u), v-\lambda_j(j,v)))\\
&\quad + 
\sum_{\scriptstyle (k,w)\in P_+ \atop
\scriptstyle
k\in (u-\lambda_-(i,u),u)\cap
 (v-\lambda_-(j,v),v)
} 
\frac{1}{2}\Bigl( b_{jk}(w) |b_{ki}(w)|
+ |b_{jk}(w)|b_{ki}(w)\Bigr)
\end{split}
\end{align}
for $(u-\lambda_-(i,u),u)
\cap(v-\lambda_-(j,v),v)\neq \emptyset$.
\end{proof}

\subsection{
Local version of constancy condition
and
second proof of 
Proposition \ref{prop:const} 
}
\label{subsec:local}

Here we give a `local' version of Proposition
\ref{prop:const}, thereby providing an alternative
 proof of Proposition \ref{prop:const}.
The formula here is parallel 
to the ones in
\cite[Lemma 6.1]{Fock09} and \cite[Lemma 2.17]{Fock07}
by Fock and Goncharov.
In fact, we have reached the formula while trying
to interpolate  Proposition \ref{prop:const} and
their results. 
An important difference from their formulas
is the use of $F$-polynomials
in \eqref{eq:W}.

Let $\mathcal{A}(B,x,y)$ be any cluster algebra.
For each seed $(B',x',y')$, we set
\begin{align}
\label{eq:W}
V'&:=\frac{1}{2}\sum_{i\in I}  F'_i \wedge (y'_i
[y'_i]_{\mathbf{T}})
=\sum_{i\in I}
F'_i \wedge y'_i
+\frac{1}{2}\sum_{i,j\in I}
b'_{ij}F'_i\wedge F'_j,
\end{align}
where $F'_i=F'_i(y)$ ($i\in I$) are the $F$-polynomials for
 $(B',x',y')$.

\begin{Proposition}[Local constancy, {cf.~\cite[Lemma 6.1]{Fock09},
\cite[Proposition 2.14 \& Lemma 2.17]{Fock07}}]
\label{prop:local}
Let $(B',x',y')$ and $(B'',x'',y'')$
be any seeds such that 
 $(B'',x'',y'')=\mu_k(B',x',y')$.
Then, we have the following relation in 
$\bigwedge ^2 \mathbb{P}_{\mathrm{univ}}(y)$.
\begin{align}
\label{eq:local}
V''-V' = y_k' \wedge (1+y'_k).
\end{align}
\end{Proposition}
\begin{proof}
Our proof is parallel to the first proof of Proposition \ref{prop:const}.
We use the following property of the corresponding $F$-polynomials,
which are derived as Lemma~\ref{lem:F}.

\par
(a) For the above $k$,
\begin{align}
\label{eq:Fi2}
\begin{split}
\textstyle
F'_kF''_k
&=
\left[
\frac{y'_k}{1+y'_k}
\right]_{\mathbf{T}}
\prod_{j:\, b'_{jk}>0} F'_j{}^{b'_{jk}}
+
\left[
\frac{1}{1+y'_k}
\right]_{\mathbf{T}}
\prod_{j:\, b'_{jk}<0} F'_j{}^{-b'_{jk}}.
\end{split}
\end{align}
\par
(b) For any $i\in I$,
\begin{align}
\label{eq:yi11}
y'_i&= [y'_i]_{\mathbf{T}}
\frac
{
\displaystyle
\prod_{j:\, b'_{ji}>0} F'_j{}^{b'_{ji}}
}
{
\displaystyle
\prod_{j:\, b'_{ji}<0} F'_j{}^{-b'_{ji}}
}
=
[y'_i]_{\mathbf{T}}\prod_{j\in I} F'_j{}^{b'_{ji}}.
\end{align}
\par
(c) For the above $k$,
\begin{align}
\label{eq:yi22}
1+y'_k&= [1+y'_k]_{\mathbf{T}}
\frac
{
F'_kF''_k
}
{
\displaystyle
\prod_{j:\, b'_{jk}<0} F'_j{}^{-b'_{jk}}
}.
\end{align}
We put \eqref{eq:yi11} and \eqref{eq:yi22} into
the right hand side of \eqref{eq:local},
expand it, then, sum them up into three parts as before.

The first part consists of the single term
$[y'_k]_{\mathbf{T}}\wedge [1+y'_k]_{\mathbf{T}}$,
which vanishes by the same reason as before.

The second part consists of the terms involving both  tropical
coefficients and $F$-polynomials.
We separate them into five parts as \eqref{eq:yF}.
By an easy calculation as before,
they are summarized as follows.
\begin{align}
\sum_{i\in I}
F''_i \wedge [y''_i]_{\mathbf{T}}
- 
\sum_{i\in I}
F'_i \wedge [y'_i]_{\mathbf{T}}.
\end{align}

The third part consists of the terms involving only
$F$-polynomials. 
We separate them into three parts as \eqref{eq:FF1}--\eqref{eq:FF3}.
Then, they are summarized as follows.
\begin{align}
\frac{1}{2}
\sum_{i\in I}
 F''_i \wedge \frac{y''_i}{[y''_i]_{\mathbf{T}}}
-
\frac{1}{2}
\sum_{i\in I}
 F'_i \wedge \frac{y'_i}{[y'_i]_{\mathbf{T}}}.
\end{align}
Therefore, we obtain the claim.
\end{proof}

Proposition \ref{prop:const}
is immediately obtained from Proposition \ref{prop:local}
as follows.

For any $u\in \mathbb{Z}$, we set
\begin{align}
V(u)&=\frac{1}{2}\sum_{i\in I}  F_i(u) \wedge (
y_i(u) [y_i(u)]_{\mathbf{T}}).
\end{align}
Note that $V(0)=0$ because $F_i(0)=1$.
Then, by Proposition \ref{prop:local}, we have,
for $u>0$,
\begin{align}
V(u) = \sum_{
\scriptstyle
(k,v)\in P_+
\atop
\scriptstyle
0\leq v < u }y_k(v) \wedge (1+y_k(v)).
\end{align}
Then, from the periodicity $V(\Omega)=V(0)=0$,
we obtain \eqref{eq:const3}.

\subsection{Skew symmetrizable case}
\label{subsec:dilog2}

For a {\em skew symmetrizable\/} matrix $B$,
the dilogarithm identities should be slightly modified.

Let $D=\mathrm{diag}(d_i)_{i\in I}$ be the (left) skew symmetrizer
of $B$, namely,
$d_i b_{ij} = - d_j b_{ji}$ holds.
Let $d=\mathrm{LCM}(d_i)_{i\in I}$ and define
$\tilde{d}_i= d/d_i \in \mathbb{N}$ ($i\in I$).
Then, $\tilde{D}=\mathrm{diag}(\tilde{d}_i)_{i\in I}$ be
 the right skew symmetrizer
 of $B$, namely,
$b_{ij}\tilde{d}_j  = -  b_{ji}\tilde{d}_i$ holds.
Let $\tilde{N}_+$ (resp. $\tilde{N}_-$) be the total number of
positive (resp. negative) monomials $[y_i(u)]_{\mathbf{T}}$
in the fundamental region $S_+$ {\em with multiplicity $\tilde{d}_i$}.
Then,  the identities \eqref{eq:DI} and  \eqref{eq:DI'}
should be modified as
\begin{align}
\label{eq:DI5}
\frac{6}{\pi^2}
\sum_{
(i,u)\in S_+
}
\tilde{d}_i L\left(
\varphi
\left(
\frac{y_i(u)}{1+y_i(u)}
\right)
\right)
&=\tilde{N}_-,\\
\label{eq:DI'6}
\frac{6}{\pi^2}
\sum_{
(i,u)\in S_+
}
\tilde{d}_i
L\left(
\varphi
\left(
\frac{1}{1+y_i(u)}
\right)
\right)
&=\tilde{N}_+,
\end{align}
and \eqref{eq:W} and \eqref{eq:local} should be modified as 
\begin{align}
\label{eq:W'}
V':=\frac{1}{2}\sum_{i\in I}  \tilde{d}_i F'_i \wedge (y'_i
[y'_i]_{\mathbf{T}})
&=\sum_{i\in I}
\tilde{d}_i F'_i \wedge y'_i
+\frac{1}{2}\sum_{i,j\in I}
b'_{ij} \tilde{d}_j F'_i\wedge F'_j,\\
\label{eq:local'}
V''-V' &= \tilde{d}_k y_k' \wedge (1+y'_k).
\end{align}

Unfortunately, for skew symmetrizable matrices
one cannot yet prove \eqref{eq:DI5}, \eqref{eq:DI'6},
and \eqref{eq:local'}.
This is because Conjecture \ref{conj:pos} (the facts (a) and (c),
in particular) is not yet proved,
and we need it for the existence of the $0/\infty$ limit and also
for $[y'_k]_{\mathbf{T}}\wedge [1+y'_k]_{\mathbf{T}}$ to vanish.
However, one can prove the rest by repeating the same argument.
To summarize, 
we have the following intermediate result
in the skew symmetrizable case.

\begin{Theorem}
\label{thm:DI2}
Suppose that Conjecture \ref{conj:pos} is true.
Then, for any  skew symmetrizable matrix,
we have \eqref{eq:local'} and 
the identities \eqref{eq:DI5} and \eqref{eq:DI'6}.
\end{Theorem}

\subsection{Concluding remark}
Let us conclude the paper by raising one natural question.

\begin{Problem}
Describe the numbers $N_{\pm}$ and $N'_{\pm}$
in  Theorems \ref{thm:DI} and \ref{thm:DI2}
by some other ways.
\end{Problem}

%\bibliography{../../biblist/biblist.bib}

\newcommand{\etalchar}[1]{$^{#1}$}
\providecommand{\bysame}{\leavevmode\hbox to3em{\hrulefill}\thinspace}
\providecommand{\MR}{\relax\ifhmode\unskip\space\fi MR }
\providecommand{\MRhref}[2]{%
  \href{http://www.ams.org/mathscinet-getitem?mr=#1}{#2}
}
\providecommand{\href}[2]{#2}

\end{document}